\documentclass[reqno, 11pt]{amsart}

\usepackage{enumitem}
\setenumerate[0]{label=(\roman*)}

\usepackage[utf8]{inputenx} 
\usepackage[T1]{fontenc}    
\usepackage[a4paper, hmargin={2.7cm,2.7cm},vmargin={3.3cm,3.3cm}]{geometry}
\usepackage{amssymb}   
\usepackage{amsthm}    
\usepackage{thmtools}  
\usepackage{mathtools} 
\usepackage{mathrsfs}  
\usepackage{commath}
\usepackage{bbm}
\usepackage[textwidth=20mm]{todonotes}
\usepackage{xcolor}

\usepackage{varioref}
\usepackage{hyperref}
\usepackage[nameinlink, capitalize, noabbrev]{cleveref}

\declaretheorem[style = plain, numberwithin = section]{theorem}
\declaretheorem[style = plain,      sibling = theorem]{corollary}
\declaretheorem[style = plain,      sibling = theorem]{lemma}
\declaretheorem[style = plain,      sibling = theorem]{proposition}
\declaretheorem[style = definition, sibling = theorem]{definition}

\declaretheorem[style = definition, sibling = theorem]{example}

\declaretheorem[style = definition, sibling = theorem]{remark}
\declaretheorem[style = definition, sibling = theorem, name={}]{pgr}

\declaretheorem[style = plain,      sibling = theorem]{question}

\newcommand{\N}{\mathbb{N}}
\newcommand{\Z}{\mathbb{Z}}

\newcommand{\R}{\mathbb{R}}

\newcommand{\T}{\mathbb{T}}

\newcommand{\Cu}{\mathrm{Cu}}

\DeclareMathOperator{\covol}{covol}

\newcommand{\ca}[1]{$\mathrm{C}^*$-algebra#1}
\newcommand{\E}{\mathcal{E}}

\def\dual#1{{#1}^\wedge}

\newcounter{IntroCount}

\newtheorem{thmIntro}[IntroCount]{Theorem}
\crefname{thmIntro}{Theorem}{Theorems}
\newtheorem{qstIntro}[IntroCount]{Question}
\crefname{qstIntro}{Question}{Questions}

\crefname{corIntro}{Corollary}{Corollaries}

\RequirePackage[backend    = biber,
                sortcites  = true,
                giveninits = true,
                doi        = false,
                isbn       = false,
                url        = false,
                maxnames = 50,
                style      = numeric-comp]{biblatex}
\DeclareNameAlias{sortname}{family-given}
\DeclareNameAlias{default}{family-given}
\title{$\mathcal{Z}$-stability of twisted group $\mathrm{C}^*$-algebras of nilpotent groups}

\author{Ulrik Enstad}
\address{Department of Mathematics,
University of Oslo,
Moltke Moes vei 35,
0851 Oslo.}
\email{ubenstad@math.uio.no}

\author{Eduard Vilalta}
\address{Department de Matemàtiques, Universitat Politècnica de Catalunya - BarcelonaTech (UPC), Diagonal 647, 08028 Barcelona}
\email[]{eduard.vilalta@upc.edu}
\urladdr{www.eduardvilalta.com}

\begin{document}

\begin{abstract}
    We prove that a twisted group \ca{} of a finitely generated nilpotent group is $\mathcal{Z}$-stable if and only if it is nowhere scattered, a condition that we characterize in terms of the given group and $2$-cocycle. As a main application, we prove new converses to the Balian--Low Theorem for projective, square-integrable representations of nilpotent Lie groups.
\end{abstract}

\maketitle

\section{Introduction}

The landmark completion of the Elliott classification program for unital separable simple nuclear \ca{s} (see among many others \cite{EGLN,TikWhiWin17,Win12Pure}) saw three regularity properties rise to prominence: \emph{$\mathcal{Z}$-stability}, a \ca{ic} analogue of von Neumann algebras' McDuffness; \emph{finite nuclear dimension}, an operator algebraic version of having finite Lebesgue dimension; and \emph{strict comparison}, a generalization of tracial comparison in II$_1$ factors. The nature of each of these properties is very different (analytic, topological and algebraic respectively) but, surprisingly, they are deeply related. Indeed, the famous Toms-Winter conjecture predicts that all three notions coincide for unital separable simple nuclear non-elementary \ca{s}, with the only remaining open implication being if strict comparison implies $\mathcal{Z}$-stability \cite{CETWW21,BBSTWW19,HiSa12,Ror04Stab,Thi20,Win12Pure,WinZacNucDim}.

While most of the study of these regularity conditions has concentrated in the simple nuclear setting (in part due to the impact and importance of the Elliott classification program), several recent works have started a large scale investigation into these properties and their potential relations in the general case \cite{AntPerThiVil24Pure,EllNiuSanTik20,RobTik17NucDim,TikWin14}. A good starting point in this analysis is the study of \ca{s} that arise from combinatorial objects such as graphs, groups and groupoids. For the case of groups, the reduced group \ca{} having finite nuclear dimension (which implies amenability of the group) has been heavily studied in the series of papers \cite{EcGilMc19FdrVirtNil,EcMcKee18,EckWu24}. At the same time, the study of which reduced group \ca{s} have strict comparison has seen several advancements (notably in the recent breakthrough \cite{AGKP24StComp}). However, for amenable groups the reduced group \ca{} is never $\mathcal{Z}$-stable, while in the non-amenable (i.e. non-nuclear) case the notion of $\mathcal{Z}$-stability loses some of its importance.

A natural generalization of group \ca{s} is that of twisted group \ca{s}, which are built not only out of a group but also a $2$-cocycle (\cref{dfn:TwistCAlg}). Unlike their untwisted counterparts, twisted group \ca{s} of amenable groups can be simple \cite{Pa89,BeOm16,BeOm18} and even fall within the scope of the Elliott classification program\footnote{That is, they can be unital separable simple nuclear $\mathcal{Z}$-stable and satisfy the UCT.} (\cite[Theorem 5.5.]{BeEnva22}). Thus, these algebras are natural candidates for which to study the relation between finite nuclear dimension, strict comparison and $\mathcal{Z}$-stability. The nuclear dimension results from the untwisted case extend (in a precise sense) to the twisted case, and so a first step towards studying regularity properties of such algebras is to determine when twisted group \ca{s} of known finite nuclear dimension are $\mathcal{Z}$-stable. Concretely, the importance of $\mathcal{Z}$-stability for twisted group \ca{s} of amenable groups is twofold: First, $\mathcal{Z}$-stability implies strict comparison, a property that allows for the study of countably generated Hilbert modules over the algebra and which has found surprising and novel applications in time-frequency analysis; see \ref{subsec:TimeFreqIntro} below. Secondly, although most current classification theorems for $^*$-homomorphisms assume the codomain algebra to be both simple and $\mathcal{Z}$-stable \cite{ClassStHomI}, it is conceivable that future results will move the assumption of simplicity to the morphism itself (in the form of fullness). Thus, nuclear $\mathcal{Z}$-stable non-simple algebras become potential codomains for subsequent classification results.

Since $\mathcal{Z}$-stability passes to ideals and quotients, it follows that no nonzero ideal-quotient (i.e. ideal of a quotient) of a $\mathcal{Z}$-stable \ca{} can be elementary. This property was termed \emph{nowhere scatteredness} in \cite{ThiVil24NSCa}, and should be viewed as the non-simple version of non-elementariness; see \ref{pgr:NS}. Thus, nowhere scatteredness is a necessary condition for $\mathcal{Z}$-stability, and it is an open problem if it is also sufficient under the presence of finite nuclear dimension and separability; see for example \cite[Paragraph~1.3]{AntPerThiVil24Pure} for a discussion on the non-simple Toms-Winter conjecture.

\begin{qstIntro}\label{qst:Intro}
    Let $A$ be a separable nowhere scattered \ca{} of finite nuclear dimension. Is $A$ $\mathcal{Z}$-stable?
\end{qstIntro}

In this paper, we study \cref{qst:Intro} for twisted group \ca{s} of \emph{nilpotent groups}, that is, groups for which the upper central series $Z_0, Z_1, Z_2, \ldots$ given recursively by $Z_0 = \{ e \}$ and $Z_{i+1}/Z_i={\rm Z}(G/Z_i)$ for $i \geq 1$ terminates at $Z_n = G$ for some finite number $n$. Concretely, we answer \cref{qst:Intro} in the affirmative for twisted group \ca{s} of finitely generated nilpotent groups\footnote{For this class of groups, finite decomposition rank was shown in \cite{EcGilMc19FdrVirtNil,EcMcKee18}.}:

\begin{thmIntro}[cf.~\ref{prp:StrIrrCharNSCA}]\label{prp:thmA}
    Let $G$ be a finitely generated nilpotent group, and let $\sigma$ be a $2$-cocycle on $G$. Then, the following conditions are equivalent:
    \begin{enumerate}
        \item $C^*_r (G,\sigma )$ is nowhere scattered;
        \item $C^*_r (G,\sigma )$ is pure;
        \item $C^*_r (G,\sigma )$ is $\mathcal{Z}$-stable;
        \item $\sigma$ is non-rational.
    \end{enumerate}    
\end{thmIntro}

Statement (iv) of \cref{prp:thmA} is a condition in terms of the $2$-cocycle which generalizes the notion of non-rationality for higher-dimensional non-commutative tori (\cref{dfn:IrrCocDec}). Loosely, a $2$-cocycle on $G$ is \emph{non-rational} if all iterated quotients of $G$ by twisted centers\footnote{As defined in the paper, the twisted center ${\rm Z}(G,\sigma)$ of a group $G$ with respect to a 2-cocycle $\sigma$ consists of those central elements $g \in G$ such that $\sigma (g,h)=\sigma (h,g)$ for all $h$ in $G$.} are infinite. For example, the condition is equivalent to $[G:{\rm Z}(G,\sigma)]=\infty$ for any abelian group; see \ref{subsec:AbelIntro} below. It is well known that $\mathcal{Z}$-stability implies pureness (a Cuntz semigroup version of $\mathcal{Z}$-stability; see \ref{pgr:Pure}), and that pureness implies nowhere scatteredness. Thus, the proof of \cref{prp:thmA} reduces to showing that (i) implies (iv), and that (iv) implies (iii). We do this in \cref{lma:NSCAimpInfInd} and \cref{prp:StrIrrCharNSCA} respectively.

With view towards future structure results on strict comparison (instead of the stronger property of $\mathcal{Z}$-stability), we also use \cite[Theorem~A]{EckWu24} to state a dimension reduction phenomenon for virtually polyclycic groups (\cref{prp:NilStCompIFFNSCa}). This shows that, to prove strict comparison of $C^*_r(G,\sigma)$, it is enough to show sufficient non-commutativity conditions of the algebra. Although we do not use it in our proof, such a result can be employed to prove the equivalence between (i) and (ii) in \cref{prp:thmA}.

In addition to the main applications outlined below, we also briefly discuss the property of nowhere scatteredness for untwisted reduced group \ca{s} in \cref{sec:GenResGroups}. As explained in \cref{prp:GNonAme}, all groups of such \ca{s} must be non-amenable. We provide in \cref{prp:PermPropCG} examples of groups $G$ such that $C^*_r(G)$ is nowhere scattered (and not simple), which include all groups of the form $T\times H$ for any discrete group $H$ and any $\mathrm{C}^*$-simple group $T$.

\subsection{\texorpdfstring{Abelian and $2$-step nilpotent groups}{Abelian and 2-step nilpotent groups}}\label{subsec:AbelIntro}

In Subsections \ref{subsec:AbGrp} and \ref{subsec:GenHeis} we specialize \cref{prp:thmA} to abelian groups and certain $2$-step nilpotent groups, such as generalized discrete Heisenberg groups. For such families of groups, a $2$-cocycle being non-rational can be neatly characterized, which leads to easily checked $2$-cocycle conditions that are equivalent to $\mathcal{Z}$-stability. (For the notation in \cref{prp:thmD}, see \ref{pgr:IndCoc}.)

\begin{thmIntro}[\ref{prp:AbeCharStComp}]\label{prp:thmC}
    Let $G$ be a finitely generated abelian group with a $2$-cocycle $\sigma$. Then, $C^*_r(G,\sigma)$ is $\mathcal{Z}$-stable if and only if $[G:\mathrm{Z}(G,\sigma)]=\infty$.
\end{thmIntro}

\begin{thmIntro}[\ref{prp:NSCAGenHeis}]\label{prp:thmD}
    Let $H=H(d_1,\ldots ,d_n)$ be a generalized discrete Heisenberg group, and let $\sigma$ be a $2$-cocycle on $H$. Then, $C^*_r(H,\sigma)$ is $\mathcal{Z}$-stable if and only if $[M : \mathrm{Z}(M ,{\rm Res}(\omega_\gamma))]=\infty$ for all $\gamma \in \dual{\mathrm{Z}(G,\sigma)}$, where $M = \ker (\varphi_{D}(\omega))$, $D=\mathrm{Z}(H)/\mathrm{Z}(H,\sigma)$, and $\omega$ is a $2$-cocycle on $H/{\rm Z}(H,\sigma)$ whose inflation is cohomologous to $\sigma$.
\end{thmIntro}

\cref{prp:thmC} recovers the results of high-dimensional non-commutative tori from \cite{MR1189278}, while the relevance of \cref{prp:thmD} lies in the fact that no generalized discrete Heisenberg group with $n\geq 2$ gives rise to a simple twisted group \ca{} \cite[Corollary~3.5]{LeePac95Twis2Step}.

\subsection{Generalized time-frequency analysis}\label{subsec:TimeFreqIntro}

We apply our main result to the field of time-frequency analysis \cite{Gr01}. Among the main objects of study in this area are Gabor systems, which are sequences of time-frequency shifts of a single window function in $L^2(\R^d)$. A central problem is to determine which lattices in phase space admit a Gabor frame ---i.e. a Gabor system that has the frame property \eqref{eq:frame}--- or a Riesz sequence \eqref{eq:riesz} (sometimes with additional regularity properties imposed on the window). Several instances of this problem have been completely or partially resolved using techniques from operator algebras \cite{Be04,JaLu20,EnsThiVil24}.

Gabor systems are special cases of so-called coherent systems arising from the representation theory of nilpotent Lie groups. The analysis of such function systems has been coined \emph{generalized time-frequency analysis} \cite{Gr21} and has recently received attention in several works; see e.g. \cite{Ou24,GrRo18,GrRoRova20,va24}.

To set the stage, let $G$ be a connected, simply connected, nilpotent Lie group and let $\pi$ be a projective irreducible unitary representation of $G$ on a Hilbert space $\mathcal{H}_\pi$. Here, \emph{projective} means that
\[ \pi(g)\pi(h) = \sigma(g,h)\pi(gh) , \quad g,h \in G, \]
where $\sigma$ is an associated continuous $2$-cocycle on $G$. We assume that $\pi$ is \emph{square-integrable}, that is, all matrix coefficients of $\pi$ lie in $L^2(G)$. Under these assumptions there exists a number $d_\pi > 0$ (depending on the Haar measure of $G$) such that
\[
\int_G | \langle \xi, \pi(g) \eta \rangle |^2 \dif{g} = d_\pi^{-1} \| \xi \|^2 \| \eta \|^2, \quad \xi, \eta \in \mathcal{H}_\pi .
\]
Given a lattice $\Gamma \subseteq G$ and a \emph{window} $\eta \in \mathcal{H}_\pi$ we can form the associated \emph{coherent system} $\pi(\Gamma) \eta = (\pi(\gamma) \eta)_{\gamma \in \Gamma}$. We say that $\pi(\Gamma)$ is a \emph{frame} if there exist constants $c, C > 0$ such that
\begin{equation}
    c \| \xi \|^2 \leq \sum_{\gamma \in \Gamma} | \langle \xi, \pi(\gamma) \eta \rangle |^2 \leq C \| \xi \|^2 \quad \text{for all } \xi \in \mathcal{H}_\pi, \label{eq:frame}
\end{equation}
and a \emph{Riesz sequence} if there exist $c, C > 0$ such that
\begin{equation}
    c \| a \|_2^2 \leq \Big\| \sum_{\gamma \in \Gamma} a_\gamma \pi(\gamma) \eta \Big\|^2 \leq C \| a \|_2^2 \quad \text{for all } a = (a_\gamma)_{\gamma \in \Gamma} \in \ell^2(\Gamma) . \label{eq:riesz}
\end{equation}
Of special interest are frames and Riesz sequences where the associated window $\eta$ is \emph{smooth} in the sense that the matrix coefficient $g \mapsto \langle \eta, \pi(g) \eta \rangle$ is a Schwartz function on $G$; cf.\ \cite{Pe94,BeEnva22}. We denote the set of these vectors by $\mathcal{H}_\pi^\infty$. For homogeneous nilpotent groups $G$ it has been proved \cite{GrRoRova20} that a lattice $\Gamma \subseteq G$ can admit a coherent frame (resp.\ a Riesz sequence) with smooth window only if $d_\pi \covol(\Gamma) < 1$ (resp.\ $d_\pi \covol(\Gamma) > 1$). This is a generalization of the \emph{Balian--Low Theorem} for Gabor frames, and is expected to hold for general nilpotent Lie groups.

Using \cref{prp:thmA}, we obtain the following partial converse to this generalized version of the Balian--Low Theorem.

\begin{thmIntro}[\ref{prp:ThmEMain}]\label{thm:gfta}
Let $\pi$ be a $\sigma$-projective, square-integrable, irreducible, unitary representation of a connected, simply connected, nilpotent Lie group. Suppose that $\Gamma$ is a lattice in $G$ such that the restriction of $\sigma$ to $\Gamma$ is non-rational. Then the following hold:
\begin{enumerate}
    \item If $d_\pi \covol(\Gamma) < 1$, then there exists $\eta \in \mathcal{H}_\pi^\infty$ such that $\pi(\Gamma) \eta$ is a frame.
    \item If $d_\pi \covol(\Gamma) > 1$, then there exists $\eta \in \mathcal{H}_\pi^\infty$ such that $\pi(\Gamma) \eta$ is a Riesz sequence.
\end{enumerate}
\end{thmIntro}

The statement of \cref{thm:gfta} above subsumes the main result of \cite{BeEnva22}, where this was shown under the stronger hypothesis that $(\Gamma,\sigma)$ satisfies Kleppner's condition ---a property equivalent to the simplicity of $\mathrm{C}^*(\Gamma,\sigma)$. It also generalizes \cite[Theorem 5.4]{JaLu20} from the setting of Gabor frames to the setting of coherent systems arising from nilpotent Lie groups.

\noindent\textbf{Acknowledgements.} UE was supported by the Research Council of Norway (Project Number 314048). EV was partially supported by MINECO (grants No.\ PID2020-113047GB-I00/AEI/10.13039/501100011033 and No.\ PID2023-147110NB-I00), by the Comissionat per Universitats i Recerca de la Generalitat de Catalunya (grant No. 2021-SGR-01015) and by the Knut and
Alice Wallenberg Foundation (KAW 2021.0140).

Part of this work was carried out during visits to the University of Potsdam, the Fields Institute for Research in Mathematical Sciences, Chalmers University of Technology and the University of Oslo. We are grateful to all the institutions for their hospitality. UE thanks Erik Bedos and Tron Omland for helpful discussions.

\noindent\textbf{Notation.} Throughout the paper, all tensor products are minimal and we denote them simply by $A\otimes B$. Most of the groups appearing in this paper (all except those in \cref{sec:AppTimFre}) are discrete, as we make explicit in every statement and definition.

\section{Nowhere scatteredness and pureness}

In this first section we recall the notions of nowhere scatteredness, pureness and strict comparison. We exhibit a relation between these properties in \cref{prp:PureFinDR}, which allows one to deduce pureness from nowhere scatteredness for virtually nilpotent groups (\cref{prp:NilStCompIFFNSCa}). Further, we study how this property behaves with tensor products and $C(X)$-algebras; see Propositions \ref{prp:TensorNSca} and \ref{prp:BundlesNSCa} respectively.

Some of the results below require Cuntz semigroup techniques. This invariant will only appear in this section, where we will use it as a tool to deduce nowhere scatteredness. If the reader wishes to focus on \cref{prp:thmA} (concretely, in (i)~implies~(iii)), the relevant parts here are the definition of nowhere scatteredness (\ref{pgr:NS}) and \cref{prp:BundlesNSCa}, neither of which uses the Cuntz semigroup. Otherwise, the reader is referred to \cite{GarPer24arX} for a modern introduction to the invariant.

\begin{pgr}[Nowhere scatteredness]\label{pgr:NS}
    As defined in the introduction, a \ca{} $A$ is said to be \emph{nowhere scattered} if no nonzero ideal-quotient of $A$ is elementary; see \cite[Theorem~3.1]{ThiVil24NSCa}. This property is well-behaved and it is preserved, for example, under inductive limits and Morita equivalences; see \cite[Section~4]{ThiVil24NSCa}. For von Neumann algebras, nowhere scatteredness is equivalent to having no type I summand (\cite[Proposition~3.4]{ThiVil24NSCa}).

    By \cite[Theorem~8.9]{ThiVil24NSCa}, a \ca{} is nowhere scattered if and only every element $[a]$ in its Cuntz semigroup $\Cu (A)$ is \emph{weakly $(2,\omega )$-divisible}, that is, if for every $\varepsilon>0$ there exists $[c_1],\ldots ,[c_n]\in \Cu (A)$ such that $2[c_l]\leq [a]$ for each $l$ and $[(a-\varepsilon)_+]\leq [c_1]+\ldots +[c_n]$. Using the same arguments as in \cite[Lemma~3.5]{ThiVil23Glimm}, one can see that every element in $\Cu (A)$ is weakly $(2,\omega )$-divisible if and only if every element of the form $[a]$ with $a\in A_+$ is.
\end{pgr}

\begin{pgr}[Pureness]\label{pgr:Pure}
    A \ca{} $A$ is \emph{pure} \cite{Win12Pure} if its Cuntz semigroup $\Cu (A)$ is
    \begin{itemize}
        \item \emph{almost divisible}, that is, for every $\varepsilon>0$ and $n\in\N$ there exists $[d]\in\Cu (A)$ such that $n[d]\leq [a]$ and $[(a-\varepsilon )_+]\leq (n+1)[d]$;
        \item \emph{almost unperforated}, that is, $[a]\leq [b]$ whenever $(m+1)[a]\leq m[b]$ for some $m\in\N$.
    \end{itemize}

    Almost unperforation of the Cuntz semigroup is equivalent to \emph{strict comparison}; see \cite[Proposition~6.2]{EllRorSan11Cone}. Here, recall that $A$ has strict comparison if, for any pair $[a],[b]\in\Cu (A)$, one has $[a]\leq [b]$ in $\Cu (A)$ whenever $a$ is in the ideal generated by $b$ and
    \[
        d_\tau (a)<d_\tau (b)
    \]
    for all $2$-quasitraces $\tau$ on $A$ such that $d_\tau (b)=1$, where $d_\tau\colon (A\otimes \mathcal{K})_+\to [0,\infty]$ is the map defined as $d_\tau (a)=\lim_n\tau (a^{1/n})$.
\end{pgr}

Under the presence of certain finite dimensionality assumptions, nowhere scatteredness is enough to deduce pureness (and thus strict comparison) of the algebra. The crucial relation when the \ca{} has finite decomposition rank is the following proposition.

\begin{theorem}\label{prp:PureFinDR}
    Let $A$ be a nowhere scattered \ca{} of finite decomposition rank. Then, $A$ is pure. In particular, $A$ has strict comparison.
\end{theorem}
\begin{proof}
    As shown in \cite[Theorem~3.1]{RobTik17NucDim}, a nowhere scattered \ca{} of nuclear dimension $m$ and without any simple purely infinite ideal-quotients is $(m,N)$-pure for some $N$, hence pure by \cite[Theorem~5.7]{AntPerThiVil24Pure}. Thus, a nowhere scattered \ca{} of finite decomposition rank is pure (because simple purely infinite algebras have infinite decomposition rank).
\end{proof}

As previously stated, the rest of this section is devoted to the study of nowhere scatteredness of two constructions: Tensor products and $C_0(X)$-algebras.

\begin{lemma}\label{prp:LmaElemTens}
    Let $A,B$ be \ca{s} with $A$ nowhere scattered and let $a\in A$ and $b\in B$ be positive elements. Then, $[a\otimes b]$ is weakly $(2,\omega )$-divisible in $\Cu (A\otimes B)$.
\end{lemma}
\begin{proof}
    The elements of the form $(a-\varepsilon)_+\otimes b$ approximate $a\otimes b$, so it is enough to show that for every $\varepsilon>0$ there exist $c_1,\ldots ,c_n\in (A\otimes B)_+$ such that $2[c_l]\leq [a\otimes b]$ and $[(a-\varepsilon)_+\otimes b]\leq [c_1]+\ldots +[c_n]$ in $\Cu (A\otimes B)$.

    Fix $\varepsilon>0$. Using that $A$ is nowhere scattered, find $d_1,\ldots ,d_n\in A_+$ such that $2[d_l]\leq [a]$ and $[(a-\varepsilon)_+]\leq [d_1]+\ldots +[d_n]$. Let $(x_k^{(l)})_k\subseteq M_{2,1}(A)$ and $(y_k)_k\subseteq M_{1,n}(A)$ be sequences such that $x_k^{(l)} a (x_k^{(l)})^*\to d_l^{\oplus 2}$ and $y_k (d_1\oplus\ldots\oplus d_n) y_k^*\to (a-\varepsilon)_+$. Set $c_l=d_l\otimes b$.
    
    Let $(e_k)_k\subseteq B$ be positive contractions such that $e_k b e_k\to b$. Consider the sequences of matrices $(r_k^{(l)}),(s_k)$ defined as $r_k (i,j)=x_k^{(l)} (i,j)\otimes e_k$ and $s_k (i,j)=y_k (i,j)\otimes e_k$. By construction, these elements satisfy $r_k^{(l)} (a\otimes b)(r_k^{(l)})^*\to c_l^{\oplus 2}$ and $s_k (c_1\oplus\ldots\oplus c_n )s_k^*\to (a-\varepsilon)_+\otimes b$. This shows that $2[c_l]\leq [a\otimes b]$ and $[(a-\varepsilon)_+\otimes b]\leq [c_1]+\ldots +[c_n]$ in $\Cu (A\otimes B)$, as desired.
\end{proof}

As defined in \cite[Definition~4.2]{KirSie16arX}, recall that a family $\mathcal{F}\subseteq A_+$ is said to be \emph{filling} if, for every hereditary sub-\ca{} $D\subseteq A$ and any ideal $I$ of $A$ such that $D\not\subseteq I$, there exists $f\in\mathcal{F}$ and $z\in A$ with $zz^*\in D$ and $z^*z=f\not\in I$.

Recall that a pair of \ca{s} $A,B$ have \emph{property (F)} \cite{Tom67PropF} if the states $\varphi\otimes\phi$ with $\varphi$ and $\phi$ pure states of $A$ and $B$ respectively separate ideals in the minimal tensor product $A\otimes B$.

\begin{proposition}\label{prp:TensorNSca}
    Let $A,B$ be a pair of \ca{s} with property (F). Assume that $A$ is nowhere scattered. Then, $A\otimes B$ is nowhere scattered.
\end{proposition}
\begin{proof}
    It follows from \cite[Lemma~6.2]{KirSie16arX} that the family $\mathcal{F}=\{ a\otimes b\mid a\in A_+,\, b\in B_+ \}$ is filling in the (minimal) tensor product\footnote{More concretely, although \cite[Lemma~6.2]{KirSie16arX} is stated under the assumption of exactness, an inspection of its proof (and the references therein) shows that the same results hold whenever $A,B$ satisfy property (F).}; see also \cite[Lemma~2.15]{BlaKir04PIHaus}.

    Now, assume for the sake of contradiction that $A\otimes B$ is not nowhere scattered. By \cite[Proposition~13.6]{GarPer24arX}, there exists a functional $\lambda$ in $\Cu (A\otimes B)$ such that $\lambda (\Cu (A\otimes B))= \mathbb{N}\cup\{\infty \}$. Let $I$ be the ideal in $A\otimes B$ corresponding to the $\Cu$-ideal $\{ x\mid \lambda (x)=0 \}$.

    Let $c\in (A\otimes B)_+$ be such that $\lambda ([c])=1$ and set $D=\overline{c(A\otimes B)c}$. Since $\mathcal{F}$ is filling, we find $a\in A_+$, $b\in B_+$ and $z\in A\otimes B$ such that $zz^*\in D$ and $z^*z=a\otimes b\not\in I$. In particular, we have $a\otimes b\precsim c$ and $\lambda ([a\otimes b])\neq 0$ and, consequently, $\lambda ([a\otimes b])=1$.

    Let $\varepsilon>0$ be such that $\lambda ([(a-\varepsilon)_+\otimes (b-\varepsilon)_+])=1$. By \cref{prp:LmaElemTens}, there exist $[c_1],\ldots ,[c_n]\in \Cu (A\otimes B)$ such that $2[c_l]\leq [a\otimes b]$ and $[(a-\varepsilon)_+\otimes (b-\varepsilon)_+]\leq [c_1]+\ldots +[c_n]$. The second inequality gives the existence of some $l\leq n$ such that $\lambda ([c_l])\neq 0$, while the first inequality implies $\lambda ([c_t])<1$ for each $t$. Thus, $0<\lambda ([c_l])<1$, a contradiction.
\end{proof}


In the study of twisted group \ca{s} we will encounter $C_0(X)$-algebras (\cite[Definition~2.1]{BlaKir04GGPBundles}); see e.g. \cref{prop:packer-raeburn}. With this in mind, we add \cref{prp:BundlesNSCa} here for future use.

\begin{proposition}\label{prp:BundlesNSCa}
Let $A$ be a $C_0(X)$-algebra. Then $A$ is nowhere scattered if and only if $A_x$ is nowhere scattered for every $x \in X$.
\end{proposition}
\begin{proof}
Since every $A_x$ is a quotient of $A$ it follows that $A_x$ is nowhere scattered for every $x$ whenever $A$ is nowhere scattered.

Conversely, assume that $A_x$ is nowhere scattered for each $x$. If $A$ is not nowhere scattered, then there exist ideals $J \subseteq I \subseteq A$ such that $I/J$ is nonzero and elementary. For every $x \in X$ the quotient $I_x / J_x$ is a quotient of $I/J$, hence it is elementary itself. Since $I/J$ is nonzero it follows that for some $x \in X$ the quotient $I_x/J_x$ must be nonzero. Hence $A_x$ is not nowhere scattered, a contradiction.
\end{proof}

\begin{remark}\label{rmk:GGP}
    A \ca{} has the \emph{Global Glimm Property} if for every $a\in A_+$ and every $\varepsilon>0$ there exists a square-zero element $r\in \overline{aAa}$ such that $(a-\varepsilon )_+$ is in the ideal generated by $r$. This condition implies nowhere scatteredness, and it is not known if both properties are equivalent; see \cite{ThiVil23Glimm} for an overview.

    Blanchard and Kirchberg showed in \cite{BlaKir04GGPBundles} that a $\mathrm{C}^*$-bundle over a finite dimensional locally compact Hausdorff space has the Global Glimm Property whenever each fiber is simple and non-elementary. In light of \cref{prp:BundlesNSCa}, we ask: Does a $C_0 (X)$-algebra have the Global Glimm Property whenever each of its fibers has it?
\end{remark}

\section{\texorpdfstring{Twisted group \ca{s}}{Twisted group C*-algebras}}\label{sec:GenResGroups}

We collect here general results on (twisted) group \ca{s} related to nowhere scatteredness and pureness. We first focus on group \ca{s} (\cref{subsec:GroupCa}) and then move to twisted group \ca{s} (\cref{subsec:TwisGroupCa}); see, for example, \cref{prp:GNonAme} and \cref{lma:NSCAimpInfInd}. These results will be used to obtain our characterization of $\mathcal{Z}$-stability in \cref{sec:TwistedCalgs}.

Let $G$ be a discrete group. Denote by $\mathrm{FC}(G)$ the \emph{FC-center} of $G$, that is, the subgroup of elements with finite conjugacy class. A \emph{$2$-cocycle} on $G$ is a map $\sigma\colon G\times G\to\mathbb{T}$ such that $\sigma (g,e)=\sigma (e,g)=1$ and 
\[
    \sigma (g,h)\sigma (gh,k)=\sigma (g,hk) \sigma (h,k)
\]
for all $g,h,k$ in $G$. To a 2-cocycle $\sigma$ we associate the \emph{$\sigma$-projective left regular representation} of $G$, which is the map $\lambda_\sigma \colon G \to \mathcal{U}(\ell^2(G))$ given on the canonical basis $\{ \delta_h : h \in G \}$ of $\ell^2(G)$ by
\[
    \lambda_\sigma (g) \delta_h = \sigma(g,h)\delta_{gh}, \quad g,h \in G.
\]
An element $g \in G$ is called \emph{$\sigma$-regular} if $\sigma(g,h) = \sigma(h,g)$ whenever $gh = hg$. By the \emph{$\sigma$-twisted center} of $G$ we mean the set $\mathrm{Z}(G,\sigma)$ of all $\sigma$-regular elements in the center of $G$. This is a subgroup of $G$, cf.\ \cite[Lemma 7.1]{Kl65}. We also denote by $\mathrm{FC}(G,\sigma)$ the set of all $\sigma$-regular elements in $\mathrm{FC}(G)$ and say that $(G,\sigma)$ satisfies \emph{Kleppner's condition} if $\mathrm{FC}(G,\sigma) = \{ e \}$.

Throughout the text, we will denote by $\widetilde{\sigma}$ the \emph{antisymmetrized form} of $\sigma$, defined by
\[
    \widetilde{\sigma}(g,h)=\sigma(g,h)\sigma(h,h^{-1}gh)^{-1}.
\]

\subsection{\texorpdfstring{Group \ca{s}}{Group C*-algebras}}\label{subsec:GroupCa} In this short subsection we state permanence properties of groups whose reduced group \ca{} is nowhere scattered, and we also note that all such groups must be non-amenable; see \cref{prp:GNonAme} and \cref{qst:NonAmImpNSCA}.

We begin with the definition of reduced twisted group \ca{s}, which gives the definition of a reduced group \ca{} when $\sigma$ is trivial.

\begin{definition}\label{dfn:TwistCAlg}
    Let $\sigma$ be a $2$-cocycle on a discrete group $G$.
    The \emph{reduced twisted group \ca{}} of $(G,\sigma)$, denoted by $C^*_r (G,\sigma)$, is the $\mathrm{C}^*$-subalgebra of $\mathcal{B}(\ell^2(G))$ generated by $\lambda_\sigma (G)$.
    
    The \emph{reduced group \ca{}} of a discrete group $G$, denoted by $C^*_r (G)$, is the reduced twisted group \ca{} algebra of $(G,1)$.
\end{definition}

\begin{example}$ $
    \begin{enumerate}
        \item Any nontrivial $\mathrm{C}^*$-simple group has a nowhere scattered reduced group \ca{}, since any simple non-elementary \ca{} is nowhere scattered \cite[Example~3.2]{ThiVil24NSCa}.
        \item Let $G$ be $\mathrm{C}^*$-simple and $H$ be any group. Assume that either $G$ or $H$ is exact. Then, $C^*_r (G\times H)$ is nowhere scattered (\cref{prp:PermPropCG}). In particular, any group of the form $\mathbb{F}_n\times H$ gives rise to a nowhere scattered algebra by \cite{Pow75Simp}.
    \end{enumerate}
\end{example}

The following proposition subsumes the permanence properties of groups whose reduced group \ca{} is nowhere scattered.

\begin{proposition}\label{prp:PermPropCG}
    Let $G,H$ be discrete groups. Then,
    \begin{itemize}
        \item[(i)] if $C^*_r(H)$ is nowhere scattered and either $G$ or $H$ is exact, $C^*_r(G\times H)$ is nowhere scattered;
        \item[(ii)] if $G$ is the inductive limit of a directed family of groups $\{ G_i\}_i$ such that each $C^*_r(G_i)$ is nowhere scattered, $C^*_r(G)$ is also nowhere scattered.
    \end{itemize}
\end{proposition}
\begin{proof}
    For (i), it is well-known that $C^*_r(G\times H)$ can be written as the minimal tensor product $C^*_r(G)\otimes C^*_r(H)$. We know from \cite[Proposition~2.17]{BlaKir04PIHaus} that the pair $C^*_r(G), C^*_r(H)$ has property (F) whenever one of the groups is exact. Thus, the result follows from \cref{prp:TensorNSca}.

    For (ii), we know from \cite{Eym64Fourier} that there exist *-homomorphisms $\rho_i\colon C^*_r(G_i)\to C^*_r(G)$ such that $C^*_r(G)=\overline{\cup_i \rho_i(C^*_r(G_i))}$. The result now follows from \cite[Proposition~4.5]{ThiVil24NSCa}.
\end{proof}

\begin{proposition}\label{prp:GNonAme}
Let $G$ be a discrete group. Then, $G$ is non-amenable whenever $C_r^*(G )$ is nowhere scattered.
\end{proposition}
\begin{proof}
If $C_r^*(G )$ is nowhere scattered, it follows from \cite[Theorem~3.1]{ThiVil24NSCa} that no hereditary sub-\ca{} of $C_r^*(G )$ admits a finite-dimensional irreducible representation.

Thus, $G$ cannot be amenable, since otherwise $C_r^*(G )$ admits a one-dimensional representation (because the trivial representation is weakly contained in the left-regular representation).
\end{proof}

We are not aware of any example of a non-amenable group whose reduced group \ca{} is not nowhere scattered. In light of this, we ask:

\begin{question}\label{qst:NonAmImpNSCA}
When does a non-amenable discrete group have a nowhere scattered reduced group \ca{}?
\end{question}

\subsection{\texorpdfstring{Twisted group \ca{s}}{Twisted group C*-algebras}}\label{subsec:TwisGroupCa} As in the previous subsection, we begin by stating permanence properties of nowhere scatteredness. We then move to the study of $2$-cocycle decompositions (\cref{dfn:CocyDecomp}), which become a crucial tool to understand nowhere scatteredness of the algebra; see \cref{lma:NSCAimpInfInd}.

First, following the proof of \cref{prp:PermPropCG} one can show, mutatis-mutandis, (i) and (ii) below.

\begin{proposition}\label{prp:PermPropCGSigm}
    Let $G,H$ be discrete groups and let $\sigma=\sigma_G\times \sigma_H$ be a $2$-cocycle on $G\times H$ where $\sigma_G,\sigma_H$ are $2$-cocycles on $G$ and $H$ respectively. Then,
    \begin{itemize}
        \item[(i)] if $C^*_r(G,\sigma_G)$ is nowhere scattered and either $G$ or $H$ is exact, $C^*_r(G\times H,\sigma)$ is nowhere scattered;
        \item[(ii)] if $G$ is the inductive limit of a directed family of groups $\{ G_i\}_i$ such that $C^*_r(G_i,\sigma_G)$ is nowhere scattered, $C^*_r(G,\sigma_G)$ is also nowhere scattered.
    \end{itemize}
\end{proposition}

Whenever $N$ is an abelian group, we denote by $\dual{N}$ its Pontryagin dual.

\begin{pgr}[Inflations and induced $2$-cocycles]\label{pgr:IndCoc}
    Given a discrete group $G$, a normal subgroup $N$ and a $2$-cocycle $\omega$ on $G/N$, recall that the \emph{inflation} of $\omega$, denoted by ${\rm Inf}(\omega)$, is the $2$-cocycle on $G$ given by $\omega\circ (\pi\times\pi)$ with $\pi\colon G\to G/N$ the quotient map.

    For a discrete group $G$ and a $2$-cocycle $\sigma$ on $G$, let $N$ be a central subgroup consisting of $\sigma$-regular elements. Then, it follows from \cite[Proposition~A2]{PacRae92StructTwist} that there exists a $2$-cocycle $\omega$ on $G/N$ such that $\sigma$ is cohomologous\footnote{Two cocycles $\sigma,\sigma'$ are said to be cohomologous if $\sigma=\sigma'\partial\phi$, where $\phi\colon G\to\mathbb{T}$ is a map such that $\phi (e)=1$ and $\partial\phi (g,h)=\phi (g)^{-1}\phi (h)^{-1}\phi (gh)$.} to ${\rm Inf}(\omega)$. As is well known to experts, recall that $\mathrm{Z}(G,\sigma)=\mathrm{Z}(G,\sigma')$ whenever $\sigma$ is cohomologous to $\sigma'$.\footnote{In fact, for any central element $g$, one has $\sigma (g,h)\sigma(h,g)^{-1}=\sigma' (g,h)(\sigma'(h,g))^{-1}$.} Further, the isomorphism class of $C_r^*(G,\sigma)$ only depends on the cohomology class of $\sigma$; see, for example, \cite[Subsection~2.7]{EckWu24}.

    Given such a $2$-cocycle $\omega$ on $G/N$ and some $\gamma\in \dual{N}$, we define the $2$-cocycle $\omega_\gamma$ in $G/N$ as
    \begin{equation}
        \omega_\gamma (g,h)=\gamma (c(g)c(h)c(gh)^{-1})\omega (g,h) \label{eq:cocycle-section}
    \end{equation}
    where $c\colon G/N\to G$ is a section of the quotient map such that $c(e)=e$ (the cohomology class of $\omega_\gamma$ does not depend on the choice of section).\footnote{That $d(g,h):=\gamma(c(g)c(h)c(gh)^{-1})$ is a $2$-cocyle (and hence so is $\omega_\gamma$) is part of the statement of \cite[Theorem~1.2]{PacRae92StructTwist}. A direct proof of this is as follows: We need to show that $d(g,h)d(gh,k)=d(h,k)d(g,hk)$. Using that $\gamma$ is a morphism, one sees that the left hand side of is equal to $\gamma (c(g)c(h)c(k)c(ghk)^{-1})$. For the right hand side, we have $d(h,k)d(g,hk)=\gamma (c(h)c(k)c(hk)^{-1}c(g)c(hk)c(ghk)^{-1})$. Using that $c(h)c(k)c(hk)^{-1}\in N$ and that $N$ is a central subgroup, we can move $c(g)$ to the left and obtain that this quantity is equal to $\gamma (c(g)c(h)c(k)c(hk)^{-1}c(hk)c(ghk)^{-1})=\gamma(c(g)c(h)c(k)c(ghk)^{-1})$, as desired.}

    Note that, whenever $g$ is central, $g$ is $\omega_\gamma$-regular if and only if
    \[
        \gamma (c(g)c(h)c(g)^{-1}c(h)^{-1})\widetilde{\omega} (g,h)=1
    \]
    for all $h \in G$.

    In most of our applications, $N$ will be the twisted center ${\rm Z}(G,\sigma)$.
\end{pgr}

\begin{theorem}[{cf. \cite[Theorem~1.2]{PacRae92StructTwist}}]\label{prop:packer-raeburn}
Let $G$ be a discrete amenable group with $2$-cocycle $\sigma$. Let $N$ be a central subgroup consisting of $\sigma$-regular elements, and let $c\colon G/N\to G$ be a section such that $c(e)=e$. Then, $\sigma$ is cohomologous to ${\rm Inf}(\omega)$ for some $2$-cocycle $\omega$ on $G/N$, and $C^*_r(G,\sigma)$ is isomorphic to a $C(\dual{N})$-algebra where the fibers are given by $C^*_r(G/N, \omega_\gamma)$ for $\gamma \in \dual{N}$.
\end{theorem}

In light of the previous result, we define the notion of \emph{$2$-cocycle decomposition}, which will play a key role in the study of nowhere scatteredness and pureness.

\begin{definition}\label{dfn:CocyDecomp}
    Let $G$ be a discrete amenable group and $\sigma$ be a $2$-cocycle on $G$. We define recursively a sequence $(\Lambda^i)_{i \in \N_{\geq 0}}$ of compact topological spaces and for each $\lambda \in \Lambda^i$ a sequence of groups $(G_\lambda^i)_{i \in \N_{\geq 0}}$ equipped with 2-cocycles $(\omega_\lambda^i)_{i \in \N_{\geq 0}}$ as follows:
    \begin{itemize}
        \item for $i=0$, set $\Lambda^0 = \{ \star \}$ to be a singleton, $G^0_\star = G$, and $\omega^0_\star = \sigma$;
        \item for $i\geq 1$, we set 
        \[
            \Lambda^{i}=\left\{
                (\lambda_0,\ldots ,\lambda_{i}) \mid 
                \lambda' = (\lambda_0,\ldots ,\lambda_{i-1})\in \Lambda^{i-1},\,\,
                \lambda_{i}\in\dual{\mathrm{Z}(G^{i-1}_{\lambda'},\omega^{i-1}_{\lambda'})}
            \right\}
        \]
        and equip it with the subspace topology inherited from the product space $\Lambda^{i-1} \times \dual{\mathrm{Z}(G^{i-1}_{\lambda'},\omega^{i-1}_{\lambda'})}$. Furthermore we set
        \[ 
            G^{i}_{\lambda}=G^{i-1}_{\lambda'}/\mathrm{Z}(G^{i-1}_{\lambda'},\omega^{i-1}_{\lambda'}), \qquad \omega^{i}_{\lambda}=\rho_{\lambda_{i}},
        \]
        where $\rho$ is a $2$-cocycle on $G^{i}_{\lambda}$ such that $\omega^{i-1}_{\lambda'}$ is cohomologous to ${\rm Inf}(\rho)$, and where $\rho_{\lambda_{i}}$ is given by \eqref{eq:cocycle-section}.
    \end{itemize}
    We refer to the sequence $(G_\lambda^i,\omega_\lambda^i)_{i \in \mathbb{N}_{\geq 0}}$ as the \emph{$2$-cocycle decomposition} of $\sigma$.\footnote{There are technically many 2-cocycle decompositions, since the choices made in this construction are not unique. However, they are all unique up to cohomology, so we tacitly think of the $2$-cocycle decomposition up to this relation.}
\end{definition}

\begin{remark}
    By construction and \autoref{prop:packer-raeburn} it follows that the 2-cocycle decomposition of $\sigma$ has the following property: For every $i \geq 1$ and $\lambda' = (\lambda_0, \ldots, \lambda_{i-1}) \in \Lambda^{i-1}$, the \ca{} $C_r^*(G_{\lambda'}^{i-1}, \omega_{\lambda'}^{i-1})$ is isomorphic to a $C(X)$-algebra with $X = \dual{\mathrm{Z}(G_{\lambda'}^{i-1},\omega_{\lambda'}^{i-1})}$ where the fibers are given by $C_r^*(G_\lambda^i,\omega_\lambda^i)$ for $\lambda \in \Lambda^i$ of the form $\lambda = (\lambda_0, \ldots, \lambda_{i-1}, \lambda_i)$ where $\lambda_i \in X$.
    
    For example, for $i=1$, we have
    \[
        \Lambda^{1}=\left\{
                (\star,\lambda_{1}) \mid 
                \lambda_{1}\in\dual{\mathrm{Z}(G^{0}_{\star},\omega^{0}_{\star})}\right\}
                =\left\{
                (\star,\lambda_{1}) \mid 
                \lambda_{1}\in\dual{\mathrm{Z}(G,\sigma)}
            \right\}\cong \dual{\mathrm{Z}(G,\sigma)}.
    \]

    Further, all the groups are $G^{1}_{\lambda}=G^{0}_{\star}/\mathrm{Z}(G^{0}_{\star},\omega^{0}_{\star})=G/Z(G,\sigma)$, while the coycles are given by
    \[
        \omega^{1}_{\lambda}=\rho_{\lambda_1}
    \]
    where $\rho$ is a $2$-cocycle on $G^{1}_{\lambda}=G/Z(G,\sigma)$ such that $\omega^{0}_{\star}=\sigma$ is cohomologous to ${\rm Inf}(\rho)$, and where $\rho_{\lambda_{1}}$ is given by \eqref{eq:cocycle-section}. In other words, at step $i=1$ the groups and cocycles that appear in the $2$-cocycle decompositions can be identified with those given by \autoref{prop:packer-raeburn} for $N=Z(G,\sigma)$.
\end{remark}

\begin{example}
We compute the 2-cocycle decomposition for the 2-cocycle $\sigma$ on $G = \Z^3$ given by
\[ \sigma((k,l,m),(k',l',m')) = e^{2\pi i ( \theta k l' + k m'/2)} \]
for some irrational number $\theta$. Then $(k,l,m) \in \mathrm{Z}(\Z^3,\sigma)$ if and only if
\[ e^{2\pi i (\theta( k l' - k'l) + (k m' - k' m)/2)} = 1 \quad \text{for all } k',l',m' \in \Z. \]
This happens precisely when $k = l = 0$ and $m \in 2 \Z$, so $\mathrm{Z}(\Z^3,\sigma) = \{ (0,0) \} \times 2\Z$. Note that the formula for $\sigma$ only depends on the class of $m$ and $m'$ in $\Z_2$, and so defines a 2-cocycle $\rho$ on $\Z^3 / \mathrm{Z}(\Z^3,\sigma) \cong \Z^2 \times \Z_2$ such that $\mathrm{Inf} (\rho) = \sigma$. Letting $c \colon \Z^2 \times \Z_2 \to \Z^3$ be the section given by $c(k,l,[m]) = (k,l,m \; \mathrm{mod} \; 2)$ and $\gamma( 0, 0, m) = e^{2\pi i \xi m/2}$ for $\xi \in [0,1)$ be a character on $\mathrm{Z}(\Z^3,\sigma)$, we have that
\[ \rho_\gamma((k,l,[m]),(k',l',[m'])) = e^{2\pi i (\xi (m \; \mathrm{mod} \; 2)(m' \; \mathrm{mod} \; 2) + \theta k l' + k m'/2)} . \]
This describes $\omega_\lambda^1$ on $G^1_\lambda=\Z^3/\mathrm{Z}(\Z^3,\sigma)$ for all $\lambda=(\star,\gamma) \in \Lambda^1$. Next we determine $\mathrm{Z}(\Z^2 \times \Z_2, \rho_\gamma)$ for all characters $\gamma$ on $\mathrm{Z}(\Z^3,\sigma)$. We have that $(k,l,[m]) \in \mathrm{Z}( \Z^2 \times \Z_2, \rho_\gamma)$ if and only if
\[ e^{2\pi i \xi (((m \; \mathrm{mod} \; 2)(m' \; \mathrm{mod} \; 2) - (m' \; \mathrm{mod} \; 2)(m \; \mathrm{mod} \; 2)) + \theta (kl' - k' l) + (k m' - k' m)/2) } = 1, \]
for all $(k',l',[m']) \in \Z^2 \times \Z_2$. This is equivalent to $(k,l,m) \in \mathrm{Z}(\Z^3,\sigma)$, i.e.,\ $k=l=0$ and $m \in 2\Z$, which means that $k=l=0$ and $[m] = [0]$ in $\Z_2$. We conclude that $\mathrm{Z}(\Z^2 \times \Z_2, \omega_\gamma) = \{ 0 \}$ for all $\gamma \in \dual{(\Z^2 \times \Z_2)}$. This means that
\[ \Lambda^2 = \{ ( \star, \lambda_1, \lambda_2) \mid \lambda_1 \in \dual{\mathrm{Z}(G,\sigma)}, \lambda_2 \in \dual{(\mathrm{Z}( \Z^2 \times \Z_2, \omega_{\lambda_1}))} \} \cong \Lambda^1 , \]
and so $\Lambda^{i+1} \cong \Lambda^1$ for all $i \geq 1$. Hence $G_\lambda^i \cong \Z^2 \times \Z_2$ for all $i \geq 1$ and $\lambda \in \Lambda^i \cong \Lambda^1$, and the corresponding 2-cocycle $\omega_\lambda^i$ equals $\omega^1_{(\star,\lambda_1)}$. In other words, the 2-cocycle decomposition of $\sigma$ stabilizes from $i=2$ onward. This always happens for abelian groups $G$; see \Cref{prop:IrrAb}.
\end{example}

\begin{definition}\label{dfn:IrrCocDec}
    Let $G$ be a discrete amenable group and let $\sigma$ be a $2$-cocycle on $G$. We say that $\sigma$ is \emph{non-rational} if
    \[
        [G^{i}_{\lambda}:\mathrm{Z}(G^{i}_{\lambda},\omega^{i}_{\lambda})] = \infty
    \]
    for all $i \in \N_{\geq 0}$ and $\lambda \in \Lambda^i$.
\end{definition}

\begin{remark}
    The name \emph{non-rational} comes from the study of high-dimensional non-commutative tori, where it is known that $C^*_r(\Z^n,\sigma)$ is nowhere scattered if and only if $\sigma$ is a so-called non-rational $2$-cocycle; see the discussion preceeding \autoref{prp:CorNCTorus} for details.
\end{remark}

\begin{proposition}\label{lma:NSCAimpInfInd}
    Let $G$ be a discrete amenable group and let $\sigma$ be a $2$-cocycle on $G$. If $C^*_r(G,\sigma)$ is nowhere scattered, then $\sigma$ is non-rational.
\end{proposition}
\begin{proof}
    Let $(G^i_\lambda ,\omega^i_\lambda)$ be the $2$-cocycle decomposition. We start proving by induction on $i \in \N_{\geq 0}$ that $C^*_r (G^{i}_{\lambda},\omega^{i}_{\lambda})$ is nowhere scattered for all $\lambda \in \Lambda^i$. For $i=0$, $C^*_r (G,\sigma)$ is nowhere scattered by assumption.

    Now fix $i \geq 1$ and assume that $C^*_r (G^{i-1}_{\lambda'},\omega^{i-1}_{\lambda'})$ is nowhere scattered for all $\lambda'\in\Lambda^{i-1}$. By the definition of our decomposition and \cref{prop:packer-raeburn}, it follows that for any $\lambda = (\lambda_0, \ldots, \lambda_i) \in\Lambda^i$ the \ca{} $C^*_r (G^{i}_{\lambda},\omega^{i}_{\lambda})$ corresponds to one of the fibers of a $C(X)$-algebra which is isomorphic to $C^*_r (G^{i-1}_{\lambda'},\omega^{i-1}_{\lambda'})$ for $\lambda'=(\lambda_0,\ldots,\lambda_{i-1}) \in \Lambda^{i-1}$, where $X = \dual{\mathrm{Z}(G_{\lambda'}^{i-1},\omega_{\lambda'}^{i-1})}$. Thus, using \cref{prp:BundlesNSCa}, we deduce that $C^*_r (G^{i}_{\lambda},\omega^{i}_{\lambda})$ is nowhere scattered.

    In particular, it follows that each $G^{i}_{\lambda}$ must be infinite, since otherwise the \ca{} $C^*_r (G^{i}_{\lambda},\omega^{i}_{\lambda})$ is finite-dimensional, hence not nowhere scattered. Thus, the statement of the proposition follows from the fact that for every $i \in \N_{\geq 0}$ and $\lambda \in \Lambda^i$ one has
    \[
        [G^{i}_{\lambda}:\mathrm{Z}(G^{i}_{\lambda},\omega^{i}_{\lambda})]
        =
        \vert G^{i+1}_{\widetilde{\lambda}}\vert
    \]
    for some $\widetilde{\lambda} \in \Lambda^{i+1}$.
\end{proof}

\begin{example}
Let $G$ be any infinite discrete amenable group equipped with a $2$-cocycle $\sigma$ such that $\mathrm{Z}(G,\sigma) = \{ e \}$ (this is the case for instance if $(G,\sigma)$ satisfies Kleppner's condition). Then $\mathrm{Z}(G_\lambda^i, \sigma_\lambda^i) = \{ e \}$ for all $i \in \N_{\geq 0}$ and $\lambda \in \Lambda^i$, so $\sigma$ is non-rational. In particular, if $G$ has trivial center, every $2$-cocycle on $G$ is non-rational. This shows that non-rationality of $\sigma$ is not a sufficient condition for the nowhere scatteredness of $C^*_r (G,\sigma)$ in general, since $G$ could be e.g.\ any infinite amenable group with trivial center.
\end{example}

\section{Finitely generated nilpotent groups}\label{sec:TwistedCalgs}

This section is devoted to the following question:

\begin{question}\label{qst:NilQst}
Let $G$ be a discrete nilpotent group. When is $C_r^*(G ,\sigma )$ nowhere scattered? When is it pure? When is it $\mathcal{Z}$-stable?
\end{question}

Under the assumption of being finitely generated, we will give a complete answer to \cref{qst:NilQst} in \cref{prp:StrIrrCharNSCA}, where we show that these properties are all equivalent and coincide with $\sigma$ being non-rational. When $G$ is abelian or is in a certain family of $2$-step nilpotent groups, we give characterizations of non-rationality of $\sigma$ that are easier to verify; see Subsections \ref{subsec:AbGrp} and \ref{subsec:GenHeis} respectively.
The reader is referred to \cite{EcMcKee18,
EcGilMc19FdrVirtNil} for an introduction to the revelant properties of virtually polycyclic groups and their \ca{s}.

One of the ingredients to deduce $\mathcal{Z}$-stability of our algebras will be the following result due to Eckhardt and Wu \cite[Theorem~A]{EckWu24} (see also \cite{EcGilMc19FdrVirtNil} and \cite[Theorem 5.5.]{BeEnva22}). We specialize it here to our case of interest:

\begin{theorem}\label{thm:BeEnva22}
    Let $G$ be a virtually polycyclic group and $\sigma$ be a $2$-cocyle on $G$. Then, $C^*_r (G,\sigma )$ has finite nuclear dimension.

    If, additionally, $G$ is nilpotent, $\mathrm{FC}(G,\sigma) = \{ e \}$ and $G$ is infinite, then $C^*_r (G,\sigma )$ is a unital, separable, simple, nuclear, $\mathcal{Z}$-stable \ca{}.
\end{theorem}

Although we will not use this in our proof of \cref{prp:thmA}, we note that one can deduce pureness (and thus strict comparison) of twisted virtually polycyclic group \ca{s} by employing the following dimension reduction phenomenon. Recall the definition of the Global Glimm Property from \cref{rmk:GGP}.

\begin{theorem}\label{prp:NilStCompIFFNSCa}
    Let $G$ be a virtually polycyclic group and $\sigma$ be a $2$-cocyle on $G$. Then, $C^*_r (G,\sigma )$ is pure if and only if $C^*_r (G,\sigma )$ has the Global Glimm Property. 

    Additionally, if $G$ is finitely generated and virtually nilpotent, then $C^*_r (G,\sigma )$ is pure if and only if $C^*_r (G,\sigma )$ is nowhere scattered.
\end{theorem}
\begin{proof}
    The first part of the statement follows from \cite[Theorem~A]{EckWu24} together with \cite[Theorem~B]{AntPerThiVil24Pure}. For the second part, use that $C^*_r(G,\sigma)$ has finite decomposition rank (\cite{EcGilMc19FdrVirtNil}) and \cref{prp:PureFinDR}.
\end{proof}


The following lemma is well-known, but we provide a proof for completeness.

\begin{lemma}\label{lem:FcCenterIsCenter}
Let $G$ be a torsion-free nilpotent group. Then $\mathrm{Z}(G) = \mathrm{FC}(G)$.
\end{lemma}

\begin{proof}
Since $G$ is a torsion-free nilpotent group, \cite[Corollary 2.22]{CleMajZym17} implies that group $G/\mathrm{Z}(G)$ is torsion-free. Hence the subgroup $\mathrm{FC}(G)/\mathrm{Z}(G)$ must be torsion-free as well.

On the other hand, it is also the case that $\mathrm{FC}(G)/\mathrm{Z}(G)$ is torsion, which we prove by induction on the length $n$ of its upper central series. If $n=0$, that is, if $G$ is the trivial group, then the statement holds trivially. Now suppose it holds for all torsion-free nilpotent groups with an upper central series of length less than $n$ and let $G$ be torsion-free nilpotent of length $n$. Let $g \in \mathrm{FC}(G)$. The group $\overline{G} = G/\mathrm{Z}(G)$ is torsion-free and nilpotent of length less than $n$ by \cite[Corollary 2.22]{CleMajZym17} and $\overline{g} = g \mathrm{Z}(G) \in \mathrm{FC}(\overline{G})$, so the induction hypothesis gives a positive natural number $k$ such that $\overline{g}^k \in \mathrm{Z}(\overline{G})$. This means that for every $h \in G$ the commutator $[g^k,h]$ lies in $\mathrm{Z}(G)$. It follows from the commutator identities (cf.\ \cite[Lemma 1.3]{CleMajZym17})
\[ [g_1,g_2g_3] = [g_1,g_3](g_3 [g_1,g_2]g_3^{-1}), \quad [g_1g_2,g_3] = (g_2[g_1,g_3]g_2^{-1})[g_1,g_3] \quad g_i \in G, \]
that the map $\phi\colon G \to G$ given by $h \mapsto [g^k,h]$ is a homomorphism. Since $g^k$ has finite conjugacy class, the index of its centralizer $C(g^k) = \{ h \in G : [g^k,h] = e \}$ in $G$ is finite. Since the centralizer coincides with the kernel of $\phi$, we conclude that $\mathrm{im} \phi \cong G/\mathrm{ker}\phi$ is finite. Using again the commutator identities, there must exist a positive integer $l$ such that
\[ e = [g^k,h^l] = [g^k,h]^l = [g^{kl},h] , \qquad h \in G . \]
We conclude that $g^{kl} \in \mathrm{Z}(G)$. Since $g \in \mathrm{FC}(G)$ was arbitrary, this shows that $\mathrm{FC}(G) / \mathrm{Z}(G)$ has torsion, which finishes the induction proof.

Since $\mathrm{FC}(G) / \mathrm{Z}(G)$ is both torsion-free and torsion, it must be trivial, so $\mathrm{FC}(G) = \mathrm{Z}(G)$.
\end{proof}

From \Cref{lem:FcCenterIsCenter} it follows that $\mathrm{Z}(G,\sigma)=\mathrm{FC}(G,\sigma)$ for any 2-cocycle $\sigma$. Thus, for any such infinite group, \cref{thm:BeEnva22} states that $C^*_r(G,\sigma)$ is simple and $\mathcal{Z}$-stable if and only if $\mathrm{Z}(G,\sigma)$ is trivial. With view towards \cref{prp:StrIrrCharNSCA}, we first show that for any infinite finitely generated nilpotent group $G$ the algebra $C^*_r(G,\sigma)$ is a finite direct sum of simple $\mathcal{Z}$-stable algebras whenever $\mathrm{Z}(G,\sigma)$ is finite; see \cref{prp:TwistCentFin}.

\begin{lemma}\label{prp:FinIndInc}
    Let $G$ be a finitely generated nilpotent group, let $\sigma$ a $2$-cocycle on $G$, and let $N\subseteq G$ a finite index normal subgroup. Then,
    \begin{enumerate}
        \item $\mathrm{Z}(G)\cap N\subseteq \mathrm{Z}(N)$ is of finite index;
        \item $\mathrm{Z}(G,\sigma )\cap N\subseteq \mathrm{Z}(N,\sigma )$ is of finite index.
    \end{enumerate}
\end{lemma}
\begin{proof}
    To prove (i), note that $\mathrm{Z}(G)\subseteq \mathrm{FC}(G)$ is of finite index (see e.g. \cite[Lemma~2.3]{EckGil16MJ}). Thus, $\mathrm{Z}(G)\cap N\subseteq \mathrm{FC}(G)\cap N$ is a finite index inclusion (since our groups are nilpotent, the quickest way to see this is to note that  $\mathrm{FC}(G)\cap N/\mathrm{Z}(G)\cap N$ is a torsion finitely generated nilpotent group, and thus finite). Further, since $N$ is of finite index, it is well-known that $\mathrm{FC}(G)\cap N=\mathrm{FC}(N)$ (see e.g. \cite[Lemma~2]{Bae48}). So, we have that
    \[
        \mathrm{Z}(G)\cap N\subseteq \mathrm{FC}(G)\cap N=\mathrm{FC}(N)
    \]
    is of finite index. Since $\mathrm{Z}(G)\cap N\subseteq \mathrm{Z}(N)\subseteq \mathrm{FC}(N)$, it follows that $\mathrm{Z}(G)\cap N\subseteq \mathrm{Z}(N)$ is also of finite index.

    For (ii), take $h\in \mathrm{Z}(N,\sigma)$ and let $g\in G$. Using (i), find $l\leq [\mathrm{Z}(N):\mathrm{Z}(G)\cap N]<\infty$ such that $h^l$ is in the center of $G$, and let $k\leq [G:N]$ be such that $g^k\in N$. Then, the element $h^{kl}$ is in the center of $G$ (since $h^l$ already is). By \cite[Lemma 7.1]{Kl65} we have that $\tilde{\sigma}$ is a bicharacter\footnote{Here, a function $\beta \colon G \times G \to \T$ is called a \emph{bicharacter} if $\beta( \cdot, h)$ and $\beta(g, \cdot)$ are characters on $G$ for all $g,h \in G$.} when restricted to an abelian subgroup. Hence, since $h^l$ is in the center of $G$, we have that
    \[
        \widetilde{\sigma}(h^{lk},g)=
        \widetilde{\sigma}((h^{l})^{k},g)=
        \widetilde{\sigma}(h^{l},g)^{k}=
        \widetilde{\sigma}(h^{l},g^{k})=1
    \]
    where in the last step we used that $g^{k}\in N$ and $h^l\in \mathrm{Z}(N,\sigma)$. This shows $h^{lk}\in \mathrm{Z}(G,\sigma)\cap N$.

    Thus, the quotient $\mathrm{Z}(N,\sigma )/\mathrm{Z}(G,\sigma)\cap N$ is a finitely generated, nilpotent, torsion group, hence finite as desired.
\end{proof}

\begin{proposition}\label{prp:TwistCentFin}
    Let $G$ be an infinite finitely generated nilpotent group, and let $\sigma$ be a $2$-cocycle. Assume that $\mathrm{Z}(G,\sigma )$ is finite. Then, $C^*_r (G,\sigma )$ is a finite direct sum of unital simple separable nuclear $\mathcal{Z}$-stable algebras.
\end{proposition}
\begin{proof}
    There exists a torsion-free, finitely generated, finite index subgroup $N$ of $G$. Indeed, by Theorem 6.6 of \cite{CleMajZym17} and the discussion afterwards, $G$ is a linear group, i.e.,\ a subgroup of $\mathrm{GL}(n,\mathbb{Z})$ for some $n$, and by Selberg's lemma \cite{Sel60}, finitely generated linear groups have a torsion-free subgroup of finite index. Furthermore, a finite index subgroup of a finitely generated group is necessarily finitely generated. It follows from \cref{prp:FinIndInc} that $\mathrm{Z}(G,\sigma)\cap N\subseteq \mathrm{Z}(N,\sigma )$ is a finite index inclusion. Since $\mathrm{Z}(G,\sigma )$ is finite, so is $\mathrm{Z}(N,\sigma)$. However, since $N$ is torsion-free, we must have that $\mathrm{Z}(N,\sigma )=\mathrm{FC}(N,\sigma)$ is trivial.

    Thus, \cref{thm:BeEnva22} implies that $C^*_r (N,\sigma)$ is unital simple separable nuclear and $\mathcal{Z}$-stable, because $N$ is infinite since $G$ is, and $N$ is nilpotent. Now, apply \cite[Lemma~4.14]{EckWu24} to deduce the desired result\footnote{Concretely, the proof of \cite[Lemma~4.14]{EckWu24} shows that $C^*_r (G,\sigma )$ is a finite direct sum of unital simple separable nuclear $\mathcal{Z}$-stable algebras whenever $C^*_r (N,\sigma)$ is itself unital simple separable nuclear $\mathcal{Z}$-stable.}.
\end{proof}

Inspired by the strategy in \cite{EcMcKee18}, the proof of \cref{prp:StrIrrCharNSCA} below will be by induction on the Hirsch length of the group. We refer again to \cite[Subsection~2.1.1]{EcMcKee18} for an introduction to this numerical invariant that takes values on $\N$. For us, the relevance of the Hirsch number $h(G)$ of a finitely generated nilpotent group comes from the following properties: $h(G)=0$ if and only if $G$ is finite; and $h(G)=h(N)+h(G/N)$ for any normal subgroup of $G$.

\begin{theorem}\label{prp:StrIrrCharNSCA}
    Let $G$ be an infinite discrete finitely generated nilpotent group, and let $\sigma$ be a $2$-cocycle on $G$. Then, the following conditions are equivalent:
    \begin{enumerate}
        \item $C^*_r (G,\sigma )$ is nowhere scattered;
        \item $C^*_r (G,\sigma )$ is pure;
        \item $C^*_r (G,\sigma )$ is $\mathcal{Z}$-stable;
        \item $\sigma$ is a non-rational $2$-cocycle;
        \item $G$ is infinite and $[G^{h-1}_\lambda : \mathrm{Z}(G^{h-1}_\lambda,\omega^{h-1}_\lambda)]=\infty$ for all $\lambda \in \Lambda^{h-1}$, where $h$ is the Hirsch length of $G$ and $(G^i_\lambda ,\omega^i_\lambda)$ is the $2$-cocycle decomposition of $(G,\sigma)$.
    \end{enumerate}
\end{theorem}
\begin{proof}
    A pure \ca{} is always nowhere scattered, so (ii)~implies~(i)\footnote{In fact, \cref{prp:NilStCompIFFNSCa} already shows that (i)~and~(ii) are equivalent.}. It is well-known that any $\mathcal{Z}$-stable \ca{} is pure (see e.g. \cite[Proposition~5.2]{AntPerThiVil24Pure}). Thus, (iii)~implies~(ii). Further, (i)~implies~(iv) by \cref{lma:NSCAimpInfInd} and (iv)~trivially implies~(v).

    We will now prove by induction on the Hirsch number $h(G)$ of $G$ that (v)~implies~(iii), which will finish the proof. We note that, since $G$ is infinite, we have $h(G)\geq 1$. Further, for $h(G)=1$, one gets $1=h(G)=h(G/\mathrm{Z}(G,\sigma))+h(\mathrm{Z}(G,\sigma))$. Since we are assuming that $G/\mathrm{Z}(G,\sigma)$ is infinite, we must have $h(\mathrm{Z}(G,\sigma))=0$ i.e. $\mathrm{Z}(G,\sigma)$ is finite. \cref{prp:TwistCentFin} now shows that $C^*_r(G,\sigma)$ is $\mathcal{Z}$-stable.

    Thus, fix $n\in\N_{>1}$ and assume that the equivalence holds whenever the Hirsch number is strictly less than $n$. Take $G$ such that $h(G)=n$. As before, one has
    \[
        h(G)=h(\mathrm{Z}(G,\sigma ))+h(G/\mathrm{Z}(G,\sigma)).
    \]

    In particular, either $h(G/\mathrm{Z}(G,\sigma))<n$ or $h(\mathrm{Z}(G,\sigma))=0$. If $h(\mathrm{Z}(G,\sigma))=0$, it follows from \cref{prp:TwistCentFin} that $C^*_r(G,\sigma)$ is $\mathcal{Z}$-stable. Otherwise we can apply our induction hypothesis to deduce that all \ca{s} of the form $C^*_r (G/\mathrm{Z}(G,\sigma),\omega^{1}_\gamma)$ for $\gamma \in \dual{\mathrm{Z}(G,\sigma)}$ are $\mathcal{Z}$-stable.

    Now use \cref{prop:packer-raeburn} to write $C^*_r(G,\sigma)$ as the section algebra of a bundle over $\dual{\mathrm{Z}(G,\sigma)}$ where, by the induction hypothesis, all fibers are $\mathcal{Z}$-stable. Furthermore, note that $\dual{\mathrm{Z}(G,\sigma)}$ is metrizable and finite dimensional: Indeed, $\mathrm{Z}(G,\sigma )$ is abelian, countable and discrete (because $G$ is), and finitely generated (because $G$ is a finitely generated nilpotent group, hence Noetherian). Thus, the properties of $\dual{\mathrm{Z}(G,\sigma)}$ together with \cite[Theorem~4.6]{HirRorWinC0XAlg} allow us to deduce that $\mathcal{Z}$-stability of the fibers passes to the whole algebra $C^*_r(G,\sigma)$.

    This finishes the induction argument, and the proof.
\end{proof}

\begin{example}
    Consider the $2$-step free nilpotent group $G(3)$, defined setwise as $\Z^6$ where multiplication is given by
    \[
        \begin{split}
            &(r_1,r_2,r_3,r_{12},r_{13},r_{23})(s_1,s_2,s_3,s_{12},s_{13},s_{23}):=\\
        &(r_1+s_1,r_2+s_2,r_3+s_3,r_{12}+s_{12}+r_1 s_2,r_{13}+s_{13}+r_1 s_3,r_{23}+s_{23}+r_2 s_3).
        \end{split}
    \]

    In \cite[Example~2.7]{Om15} a formula is given for all $2$-cocycles of $G(3)$ up to similarity. Using it, one can produce examples of $2$-cocycles that are non-rational but which do not satisfy Kleppner's condition: Let $\theta$ be an irrational number and let $\sigma$ be the $2$-cocycle given by 
    \[
        \sigma (r,s)=e^{2\pi i\theta \left(s_{13}r_2+s_3(r_1 r_2-r_{12})+s_{13}r_1+\frac{1}{2}s_3 r_1 (r_1 -1) +r_3 (s_{13}+r_1 s_3)+\frac{1}{2}r_1 s_3 (s_3-1)\right)}
    .
    \]

    It is readily checked that central elements of $G(3)$ are of the form $(0,0,0,r_{12},r_{13},r_{23})$. Thus, if $r$ is central, one has
    \[
        \widetilde{\sigma} (r,s)=e^{2\pi i\theta(-s_1r_{13}-s_2r_{13}-s_3 (r_{12}+r_{13}))}.
    \]

    Note that $\mathrm{Z}(G(3),\sigma)$ consists of the elements of the form $(0,0,0,0,0,r_{23})$. In particular, $C^*_r(G(3),\sigma)$ is not simple. Further, direct computations show that the center of $G(3)/\mathrm{Z}(G(3),\sigma)$ coincides with $\mathrm{Z}(G(3))/\mathrm{Z}(G(3),\sigma)$. In particular, the section $c\colon G(3)/\mathrm{Z}(G(3),\sigma)\to G(3)$ given by 
    \[
        [r_1,r_2,r_3,r_{12},r_{13},r_{23}]\mapsto (r_1,r_2,r_3,r_{12},r_{13},0)
    \]
    is well-defined and sends central elements to central elements.

    Now, given any $\gamma\in\dual{\mathrm{Z}(G(3),\sigma)}$ and any $[s]$ central, we get
    \[
        \gamma (c([s])c([t])c([s])^{-1}c([t])^{-1})=1
    \]
    which implies that, for $\omega$ as in \cref{pgr:IndCoc}, $\mathrm{Z}(G(3)/\mathrm{Z}(G(3),\sigma),\omega)=\mathrm{Z}(G(3)/\mathrm{Z}(G(3),\sigma),\omega_\gamma)$ for all $\gamma$.

    Finally, for $[r]$ central, one has
    \[
        \widetilde{\omega} ([r],[s])=\widetilde{\sigma} (r,s)
    \]
    and so $[r]$ is $\omega$-regular if and only if $r$ is $\sigma$-regular, which happens if and only if $r_{12}=r_{13}=0$, that is, $[r]=0$.

    Note that $G(3)/\mathrm{Z}(G(3),\sigma)$ is a torsion-free finitely generated nilpotent group, so it follows that $\mathrm{FC}(G(3)/\mathrm{Z}(G(3),\sigma),\omega_\gamma)=\mathrm{Z}(G(3)/\mathrm{Z}(G(3),\sigma),\omega_\gamma)=\{ 0\}$.

    In other words, the $2$-cocycle decomposition stabilizes at the first step, and it follows from our computations that $\sigma$ is non-rational. Thus, $C^*_r(G(3),\sigma)$ is $\mathcal{Z}$-stable by \cref{prp:StrIrrCharNSCA}.
\end{example}

\begin{question}
    In light of \cref{prp:StrIrrCharNSCA}, many natural questions arise. For example, does the equivalence between~(i) and~(ii) hold when $G$ is not assumed to be finitely generated? For what other groups can one get a similar result? Natural candidates are finitely generated virtually polycyclic groups (which are known to have finite nuclear dimension by \cite{EckWu24}), and finitely generated groups where no non-trivial quotient is finite \cite{JusMon13}.
\end{question}

\subsection{Abelian groups}\label{subsec:AbGrp} We now specialize \cref{prp:StrIrrCharNSCA} to the case when $G$ is abelian, which allows us to obtain an easily checked group condition that characterizes $\mathcal{Z}$-stability; see \cref{prp:AbeCharStComp}. This result recovers the known characterization for higher-dimensional non-commutative tori obtained in \cite{EnsThiVil24} (\cref{prp:CorNCTorus}) and can be used to study $2$-cocycles on products (\cref{prp:Products}).

\begin{lemma}\label{prop:IrrAb}
Let $G$ be an abelian group and let $\sigma$ be a $2$-cocycle. Then, $\sigma$ is a non-rational $2$-cocycle if and only if $[G:\mathrm{Z}(G,\sigma)]=\infty$.
\end{lemma}
\begin{proof}
Retain the notation from \ref{pgr:IndCoc} with $N=Z(G,\sigma)$. It follows from the discussion in that paragraph that $g\in \mathrm{Z}(G/Z(G,\sigma))$ is $\omega_\gamma$-regular if and only if 
    \[
        \gamma (c(g)c(h)c(g)^{-1}c(h)^{-1})\widetilde{\omega} (g,h)=1
    \]
    for all $h$, where $c\colon G/Z(G,\sigma)\to G$ is a section.

Note that all commutators are trivial, which implies that $s\in \mathrm{Z}(G/Z(G,\sigma))$ is $\omega_\gamma$-regular if and only if it is $\omega$-regular. Let $g_0\in G$ be such that $\pi (g_0)=g$. Since $\sigma$ is cohomologous to ${\rm Inf}(\omega)$, it follows that $g$ is $\omega$-regular if and only if $g_0$ is $\sigma$-regular. In other words, $g=e$ and thus $\mathrm{Z}(G^1_\lambda,\omega_\gamma)=\{ e\}$ for each $\gamma$, where recall that $G^1_\lambda=G/Z(G,\sigma)$ for all $\lambda$. By construction of our decomposition, this implies $[G^{i-1}_\lambda : \mathrm{Z}(G^{i-1}_\lambda,\omega^{i-1}_\lambda)]=\infty$ for all $\lambda$ and $i>1$, as desired.
\end{proof}

\begin{theorem}\label{prp:AbeCharStComp}
    Let $G$ be a finitely generated abelian group with a $2$-cocycle $\sigma$. Then, the following are equivalent:
    \begin{itemize}
        \item[(i)] $C^*_r(G,\sigma)$ is nowhere scattered;
        \item[(ii)] $C^*_r(G,\sigma)$ is $\mathcal{Z}$-stable;
        \item[(iii)] $[G:\mathrm{Z}(G,\sigma)]=\infty$.
    \end{itemize}
\end{theorem}
\begin{proof}
    (iii)~is equivalent to non-rationality of $\sigma$ by \cref{prop:IrrAb}. Thus, the result follows from \cref{prp:StrIrrCharNSCA}.
\end{proof}

\begin{remark}\label{prop:field-ns}
    Using a proof similar to that of \cref{prp:StrIrrCharNSCA} and the fact that \cref{prp:BundlesNSCa} has no assumptions on the base space, one sees that the equivalence (i)$\iff$(iii) holds for any abelian group (finitely generated or not).
\end{remark}

As stated, our result can be used to reprove the characterization of when a higher-dimensional non-commutative torus is $\mathcal{Z}$-stable. Recall that, given a skew-symmetric, real $n \times n$ matrix $\theta = (\theta_{i,j})_{i,j=1}^n$, the associated noncommutative torus $A_{\theta}$ is the \ca{} universally generated by unitary elements $u_1, \ldots, u_n$ satisfying the relations
\[ u_j u_i = e^{2\pi i \theta_{i,j}} u_i u_j , \quad 1 \leq i,j \leq n. \]
Furthermore, $A_\theta$ is called \emph{non-rational} if at least one of the entries of $\theta$ is irrational.

\begin{corollary}\label{prp:CorNCTorus}
    A noncommutative torus is $\mathcal{Z}$-stable if and only if it is non-rational.
\end{corollary}
\begin{proof}
    We have that $A_\theta$ is isomorphic to the twisted group \ca{} $C_r^*(\Z^n,\sigma)$ where
    \[ \sigma(k,l) = e^{-\pi i \sum_{1 \leq r < s \leq n} k_r \theta_{r,s} l_s} , \quad k,l \in \Z^n . \]
    (Here, the unitary $u_j$ corresponds to $\lambda_\sigma(e_j)$ where $e_j$ denotes the $j$th basis vector of $\Z^n$.) Now $\mathrm{Z}(\Z^n,\sigma)$ has finite index in $\Z^n$ if and only if for each $1 \leq i \leq n$ there is some $m_i > 0$ such that $m_i e_i \in \mathrm{Z}(\Z^n,\sigma)$. Since $\widetilde{\sigma}(e_i,e_j) = e^{-2\pi i \theta_{i,j}}$, this means that $\mathrm{Z}(\Z^n,\sigma)$ has finite index if and only if each number $\theta_{i,j}$ is rational. Hence the result follows from \cref{prp:AbeCharStComp}.
\end{proof}

Another application of \cref{prp:AbeCharStComp} is the study of $2$-cocycles in products of abelian groups. As seen in \cref{prp:PermPropCGSigm}, $C_r^*(G_1\times G_2, \sigma_1\times \sigma_2)$ is nowhere scattered whenever $C_r^*(G_1, \sigma_1)$ is and $G_1$ or $G_2$ is exact. However, not all $2$-cocycles on $G_1\times G_2$ are of the form $\sigma_1\times\sigma_2$: Following the notation of \cite[Section~3]{Oml14Primeness}, every $2$-cocycle on $G_1\times G_2$ is cohomologous to a $2$-cocycle of the form
\begin{equation}\label{eq:SigmaProd}
    \sigma ((g_1,g_2),(h_1,h_2))=\sigma_1(g_1,h_1)\sigma_2(g_2,h_2)f(h_1,g_2),
\end{equation}
where $\sigma_1(g_1,h_1):=\sigma ((g_1,e),(h_1,e))$, $\sigma_2(g_2,h_2):=\sigma ((e,g_2),(e,h_2))$ and $f\colon G_1\times G_2\to\mathbb{T}$ is a bihomomorphism. 

\begin{proposition}\label{prp:Products}
    For $i=1,2$, let $G_i$ be a discrete abelian group and $\sigma_i$ be a $2$-cocycle on $G_i$. Let $f$ be a bihomomorphism on $G_1\times G_2$, and set $\sigma$ as in (\ref{eq:SigmaProd}). If for every $(g_1,g_2)\in \mathrm{Z}(G_1\times G_2,\sigma)$ one has $f(h_1,g_2)=1$ for each $h_1\in G_1$, then
    \[
        C^*_r(G_1,\sigma_1)
        \text{ is $\mathcal{Z}$-stable }
        \implies
        C^*_r(G_1\times G_2,\sigma)\text{ is $\mathcal{Z}$-stable.}
    \]

    Conversely, if $f(g_1,h_2)=f(h_1,g_2)=1$ for every $(g_1,g_2)\in \mathrm{Z}(G_1,\sigma_1)\times \mathrm{Z}(G_2,\sigma_2)$ and every $(h_1,h_2)\in G_1\times G_2$, then
    \[
        C^*_r(G_1\times G_2,\sigma)\text{ is $\mathcal{Z}$-stable }
        \implies 
        C^*_r(G_1,\sigma_1)\text{ or }C^*_r(G_2,\sigma_2)
        \text{ is $\mathcal{Z}$-stable.}
    \]
\end{proposition}
\begin{proof}
    The first condition gives, by \cite[Lemma~3.3]{Oml14Primeness}, that $\pi_1 (\mathrm{Z}(G_1\times G_2,\sigma))\subseteq \mathrm{Z}(G_1,\sigma_1)$, while the second condition gives (also by the same lemma) that $\mathrm{Z}(G_1,\sigma_1)\times \mathrm{Z}(G_2,\sigma_2)\subseteq \mathrm{Z}(G_1\times G_2,\sigma)$.

    Thus, one has
    \[
        \begin{split}
            [G_1\times G_2:\mathrm{Z}(G_1\times G_2,\sigma)]&\geq [G_1:\mathrm{Z}(G_1,\sigma_1)]\quad\quad\quad\quad\quad\text{ and }\\
            [G_1:\mathrm{Z}(G_1,\sigma_1)][G_2:\mathrm{Z}(G_2,\sigma_2)]&\geq [G_1\times G_2:\mathrm{Z}(G_1\times G_2,\sigma)]
        \end{split}
    \]
    respectively.
    
    The result now follows from \cref{prp:AbeCharStComp}.
\end{proof}

\subsection{\texorpdfstring{$2$-step nilpotent groups and generalized discrete Heisenberg groups}{2-step nilpotent groups and generalized discrete Heisenberg groups}}\label{subsec:GenHeis} In this subsection, we give a sufficient condition for nowhere scatteredness of twisted group \ca{s} of $2$-step nilpotent groups, which we then use to obtain a sufficient criterion for $\mathcal{Z}$-stability in the finitely generated case; see \cref{prp:CharNSCAforSome2step}. This condition is easier to check than that of \cref{prp:StrIrrCharNSCA}, and we use it to do an in-depth study of generalized discrete Heisenberg groups (\cref{pgr:GenDisHeis}).

For a normal subgroup $N$ of $G$ and a $2$-cocycle $\sigma$ on $G$, we denote by ${\rm Res}\,\sigma$ the \emph{restriction} of $\sigma$ to $N$. 

Now let $G$ be a discrete $2$-step nilpotent group, let $\sigma$ be a $2$-cocycle on $G$, and let $S$ be a central subgroup of $G$. Following \cite[Section~2]{LeePac95Twis2Step}, denote by $\varphi_S (\sigma)\colon G\to\dual{S}$ the map given by
\[
    \varphi_S (\sigma)(g)(d):=\sigma (d,g)\sigma (g,d)^{-1}
    =\widetilde{\sigma}(d,g).
\]
As observed on p.\ 96 of \cite{LeePac95Twis2Step}, $\varphi_S$ is a group homomorphism.

We will denote by $C(G)$ the commutator subgroup of $G$, which is central because $G$ is $2$-step nilpotent.

Recall that the Hirsch number of a finitely generated $2$-step nilpotent group is always greater than or equal to $3$. The contribution of \cref{prp:CharNSCAforSome2step} is that it allows us to restrict to the study of subgroups of $G/\mathrm{Z}(G,\sigma)$, instead of going to the iterated quotients that appear in \cref{prp:StrIrrCharNSCA}.

\begin{theorem}\label{prp:CharNSCAforSome2step}
    Let $G$ be a discrete $2$-step nilpotent group and let $\sigma$ be a $2$-cocycle on $G$. Let $\omega$ be a $2$-cocycle on $G / \mathrm{Z}(G,\sigma)$ such that $\sigma$ is cohomologous ${\rm Inf}(\omega )$. Assume that there exists a subgroup $D\subseteq \mathrm{Z}(G)/\mathrm{Z}(G,\sigma )$ such that
    \begin{itemize}
        \item[(i)] $C(G/\mathrm{Z}(G,\sigma))\subseteq D$;
        \item[(ii)] $\omega$ is trivial on $D\times D$;
        \item[(iii)] the map $\varphi_{D}(\omega )\colon G/\mathrm{Z}(G,\sigma)\to \dual{D}$ is surjective.
    \end{itemize}

    Then, $C_r^*(G,\sigma)$ is nowhere scattered if and only if
    \[
        [M : \mathrm{Z}(M ,{\rm Res}(\omega_\gamma))]=\infty
    \]
    for all $\gamma \in \dual{\mathrm{Z}(G,\sigma)}$, where $M = \ker (\varphi_{D}(\omega))$.

    Further, if $G$ is finitely generated, this condition holds if and only if $C^*_r (G,\sigma)$ is $\mathcal{Z}$-stable.
\end{theorem}
\begin{proof}    
    We know from \cite[Proposition~1.6,~Theorem~1.8]{LeePac95Twis2Step} that $C^*_r (G,\sigma )$ can be written as a $C_0 (\dual{\mathrm{Z}(G,\sigma)})$-algebra with fibers isomorphic to $C^*_r (M/D,\tau_\gamma)\otimes \mathcal{K}(L^2 (G/M))$, where $M$ is abelian and $\tau_\gamma$ is defined below. Thus, it follows from \cref{prp:BundlesNSCa} and \cref{prop:field-ns} that $C^*_r (G,\sigma )$ is nowhere scattered if and only if $[M/D:\mathrm{Z}(M/D,\tau_\gamma)]=\infty$ for every $\gamma$. We will now show that $[M/D:\mathrm{Z}(M/D,\tau_\gamma)]=[M : \mathrm{Z}(M ,{\rm Res}(\omega_\gamma))]$, which will finish the proof.

    Fix $\gamma \in \dual{\mathrm{Z}(G,\sigma)}$. Then, the $2$-cocycle $\tau_\gamma$ is given by
    \[
       \tau_\gamma (g,h) = \omega_\gamma \left(c(g),c(h)\right)\omega_\gamma \left(c(g) c(h) c(gh)^{-1},c(gh)\right)^{-1}
    \]
    where $c\colon M/D\to M$ is a section that maps the identity to the identity.

    First, write
    \[
        \widetilde{\tau_\gamma} (g,h) = 
        \widetilde{\omega_\gamma} \left(c(g) ,c(h)\right)\omega_\gamma \left(c(g) c(h) c (gh)^{-1},c(gh)\right)^{-1}
        \omega_\gamma \left(c(h)c(g)c(hg)^{-1},c(hg)\right)
    \]
    and, since $M$ (and $M/D$) is abelian, we have
    \[
        \begin{split}
        \widetilde{\tau_\gamma} (g,h) &= 
        \widetilde{\omega_\gamma} \left(c(g) ,c(h)\right)
        \omega_\gamma \left(c(g) c(h) c(gh)^{-1},c(gh)\right)^{-1}
        \omega_\gamma \left(c(g)c(h)c(gh)^{-1},c(gh)\right)\\
        &=
        \widetilde{\omega_\gamma} \left(c(g) ,c(h)\right).
        \end{split}
    \]

    Now let $g,h\in M/D$ and $d,d'\in D$. Set $m=c(h)d$ and $m'=c(g)d'$. Using that $M=\ker (\varphi_{D} (\omega_\gamma))$ one gets
    \[
        \begin{split}
        \widetilde{\omega_\gamma}\left(m,m'\right)&=
        \widetilde{\omega_\gamma}\left(c(h),m'\right)
        \widetilde{\omega_\gamma}\left(d,m'\right)=
        \widetilde{\omega_\gamma}\left(c(h),c(g)\right)
        \widetilde{\omega_\gamma}\left(c(h),d'\right)
        \widetilde{\omega_\gamma}\left(d,c(g)\right)
        \widetilde{\omega_\gamma}\left(d,d'\right)\\
        &= \widetilde{\tau_\gamma}\left( h,g\right),
        \end{split}
    \]
    which shows that 
    \[
         \mathrm{Z}(M/D,\tau_\gamma) =
         \{ [m]\in M/D \mid \widetilde{\omega_\gamma} (m,m')=1 \,\,\,\forall m'\in M \}=
         \mathrm{Z}(M,\omega_\gamma)/D.
    \]
    
    Thus, we have $[M/D:\mathrm{Z}(M/D,\tau_\gamma)]=[M : \mathrm{Z}(M ,{\rm Res}(\omega_\gamma))]$, as desired.
\end{proof}

\begin{pgr}[Generalized discrete Heisenberg groups]\label{pgr:GenDisHeis}
    Let $B\in M_m (\mathbb{Z})$. Following \cite[Section~2]{LeePac95Twis2Step}, define the \emph{generalized discrete Heisenberg groups} $H(B)$ as $H(B)=\mathbb{Z}\times\mathbb{Z}^{m}\times\mathbb{Z}^{m}$ with product
    \[
        (r,s,t)(r',s',t')=
        \left(
        r+r'+tB(s')^T,
        s+s',
        t+t'
        \right).
    \]

    When $B=D$ is a diagonal matrix with values $d_1,\ldots ,d_m$, we will either write $H(D)$ or $H(d_1,\ldots ,d_m)$. As explained in the beginning of \cite[Section~2]{LeePac95Twis2Step}, every group $H(B)$ is of the form $H(d_1,\ldots ,d_n)\times \mathbb{Z}^{2(m-n)}$ for $n$ the rank of $B$ and $d_1,\ldots ,d_n\in\mathbb{Z}$.

    By \cite[Corollary~3.5]{LeePac95Twis2Step}, no twisted group \ca{} of the form $C^*_r(H(d_1,\ldots ,d_n),\sigma)$ is simple for $n\geq 2$. We will show in \cref{prp:NSCAGenHeis} below that many are still $\mathcal{Z}$-stable.
\end{pgr}

\begin{example}
    Set $d_1=1$ and take $d_2,p$ coprime. Consider the $2$-cocycle $\sigma$ on $H=H(1,d_2)$ given by
    \[
        \sigma ((r,s,t),(r',s',t')) = 
        \left( e^{2\pi i\frac{p}{d_2}}\right)^{s_1'r+t_1 s_1'(s_1'-1)/2}
        \left( e^{2\pi i\theta}\right)^{s_2't_2}
    \]
    for some fixed irrational $\theta$ (that this is a $2$-cocycle follows from the standard parametrization of $2$-cocycles on generalized discrete Heisenberg groups in \cite[Proposition~2.13]{LeePac95Twis2Step}).

    One computes $\mathrm{Z}(H)=C(H)=\{(r,0,0)\mid r\in\mathbb{Z} \}$ and $\mathrm{Z}(H,\sigma)=\{ (r,0,0)\mid r\in d_2\mathbb{Z}\}$, where recall that $C(H)$ denotes the commutator subgroup of $H$. Thus, $H/\mathrm{Z}(H,\sigma)=\{ [r,s,t]\mid r\in\mathbb{Z}_{d_2}\}$ and $C(H/\mathrm{Z}(H,\sigma))=\{[r,0,0]\mid r\in\mathbb{Z}_{d_2}\}$.

    Set $\omega=\sigma\circ (c\times c)$ where $c$ is the cross section $[r,s,t]\mapsto (r \; \mathrm{mod} \; d_2,s,t)$, and let $D=C(H/\mathrm{Z}(H,\sigma))$. Note that (i)-(iii) in \cref{prp:CharNSCAforSome2step} are satisfied. Then,
    \[
        \varphi_{D}(\omega)([r,s,t])([r',0,0])=
        e^{-2\pi i\frac{p}{d_2}s_1 r'}
    \]
    which implies $M=\{ [r,s,t]\mid r\in\mathbb{Z}_{d_2},\, 
    s_1\in d_2\mathbb{Z}\}$ where $M=\ker (\varphi_{D}(\omega))$.

    Let $\gamma\in \dual{\mathrm{Z}(H,\sigma)}\cong\mathbb{T}$. A direct computation gives
    \[
        \widetilde{\omega_\gamma} ([r,s,t],[r',s',t'])
        =
        \gamma \left((t_1s_1'-t_1's_1+d_2(t_2s_2'-t_2's_2),0,0)\right)e^{2\pi i\theta (s_2' t_2 -s_2 t_2')}
    \]
    for $[r,s,t],[r',s',t']\in M$.

    Let $\xi$ be such that $\gamma ((r,0,0))=e^{2\pi i\xi}$. By studying the cases of $\xi$ being irrational and rational separately, one sees that (in both cases) $[M:\mathrm{Z}(M,{\rm Res}(\omega_\gamma))]=\infty$. By \cref{prp:CharNSCAforSome2step}, $C^*_r(H,\sigma)$ is $\mathcal{Z}$-stable.
\end{example}

\begin{theorem}\label{prp:NSCAGenHeis}
    Let $H=H(d_1,\ldots ,d_n)$ be a generalized discrete Heisenberg group, and let $\sigma$ be a $2$-cocycle on $H$. Then, the following are equivalent:
    \begin{enumerate}
        \item $C^*_r(H,\sigma)$ is nowhere scattered;
        \item $C^*_r(H,\sigma)$ is $\mathcal{Z}$-stable;
        \item $[M : \mathrm{Z}(M ,{\rm Res}(\omega_\gamma))]=\infty$ for all $\gamma \in \dual{\mathrm{Z}(H,\sigma)}$, where $M = \ker (\varphi_{D}(\omega))$ and $D=\mathrm{Z}(H)/\mathrm{Z}(H,\sigma)$.
    \end{enumerate}
\end{theorem}
\begin{proof}
    (i)~and~(ii) are equivalent by \cref{prp:StrIrrCharNSCA}. To see that (iii)~and~(i) are equivalent, one needs to check that $D=\mathrm{Z}(H)/\mathrm{Z}(H,\sigma)$ satisfies the conditions in \cref{prp:CharNSCAforSome2step}.

    This is done in the proof of \cite[Theorem~3.4]{LeePac95Twis2Step}, where the authors use a precise description of all $2$-cocycles up to cohomology (see \cite[Proposition~2.13]{LeePac95Twis2Step}).
\end{proof}

Combining our previous results on abelian groups (\cref{prp:AbeCharStComp}) and generalized Heisenberg groups (\cref{prp:NSCAGenHeis}), we obtain the following consequence of \cref{prp:PermPropCGSigm}.

\begin{corollary}
    Let $B\in M_m (\mathbb{Z})$ and consider the group $H(B)\cong H(d_1,\ldots ,d_n)\times \mathbb{Z}^{2(m-n)}$. Let $\sigma$ be a $2$-cocycle on $H(B)$ that can be expressed as $\sigma_1\times\sigma_2$. Then, $C^*_r(H(B),\sigma)$ is $\mathcal{Z}$-stable if and only if $[\mathbb{Z}^{2(m-n)}\colon \mathrm{Z}(\mathbb{Z}^{2(m-n)},\sigma_2)]=\infty$ or $\sigma_1$ satisfies the condition in \cref{prp:NSCAGenHeis}.
\end{corollary}

\section{Applications to generalized time-frequency analysis}\label{sec:AppTimFre}

The following section is dedicated to the proof of \Cref{thm:gfta} and hence we retain the notation from \ref{subsec:TimeFreqIntro}. Namely, we let $G$ denote a connected, simply connected, nilpotent Lie group with a continuous $2$-cocycle $\sigma$, and we let $\pi \colon G \to \mathcal{U}(\mathcal{H}_\pi)$ be a $\sigma$-projective, square-integrable, irreducible, unitary representation of $G$.

\begin{theorem}\label{prp:ThmEMain}
Let $\Gamma$ be a lattice in $G$ such that the restriction of $\sigma$ to $\Gamma$ has is a non-rational $2$-cocycle. Then the following hold:
\begin{enumerate}
    \item If $d_\pi \covol(\Gamma) < 1$, then there exists a smooth vector $\eta \in \mathcal{H}_\pi$ such that $\pi(\Gamma) \eta$ is a frame.
    \item If $d_\pi \covol(\Gamma) > 1$, then there exists a smooth vector $\eta \in \mathcal{H}_\pi$ such that $\pi(\Gamma) \eta$ is a Riesz sequence.
\end{enumerate}
\end{theorem}

\begin{proof}
Let $\mathcal{E}_\Gamma$ denote the finitely generated Hilbert $\mathrm{C}^*$-module over $\mathrm{C}_r^*(\Gamma,\sigma)$ constructed in \cite[Theorem 6.5]{BeEnva22}, which satisfies that 
\[ \tau(\E_\Gamma) = d_\pi \covol(\Gamma) \]
for $\tau$ the canonical tracial state on $\mathrm{C}_r^*(\Gamma,\sigma)$.

Now $\Gamma$ is a torsion-free, finitely generated, nilpotent group and thus has polynomial growth by \cite{Gro81Poly}. Hence, if $\tau'$ is any other tracial state on $C_r^*(\Gamma,\sigma)$, then $\tau'(\E_\Gamma) = d_\pi \covol(\Gamma)$ as well by \cite[Theorem 1]{Ji95Trace}. Also, since the restriction of $\sigma$ to $\Gamma$ is non-rational, $C_r^*(\Gamma,\sigma)$ is $\mathcal{Z}$-stable by \Cref{prp:StrIrrCharNSCA}, hence has strict comparison. We can therefore argue as in the proof of \cite[Theorem 6.6]{BeEnva22} and deduce from $d_\pi \covol(\Gamma) < 1$ (resp.\ $d_\pi \covol(\Gamma) > 1$) the existence of $\eta \in \mathcal{H}_\pi^\infty$ such that $\pi(\Gamma) \eta$ is a frame (resp.\ Riesz sequence).
\end{proof}

When $\sigma$ is not non-rational, we can still obtain a quantitative upper bound on the number of smooth windows needed for a multiwindow coherent frame over a lattice $\Gamma \subseteq G$ satisfying $d_\pi \covol(\Gamma) < 1$. The following result is based on \cite[Theorem~5.9]{EckWu24}.

\begin{theorem}
Let $\Gamma$ be a lattice in $G$ such that $d_\pi \covol(\Gamma) < 1$. Denote by $H_2(\Gamma)$ the second cohomology group of $\Gamma$, and set
\[ m = f(h(H_2(\Gamma)) + h(\Gamma)) \]
where $f \colon \N \to \N$ is the function recursively defined by $f(1) = 1$ and
\[ f(n+1) = 9^n (n+1) (f(n)+1)-1 , \quad n \geq 2 .\]
Then there exist $\eta_1, \ldots, \eta_{m+1} \in \mathcal{H}_\pi^\infty$ such that $(\pi(\gamma)\eta_j)_{\gamma \in \Gamma, 1 \leq j \leq m+1}$ is a frame.
\end{theorem}
\begin{proof}
    It follows from \cite[Theorem~5.9]{EckWu24} that $A = \mathrm{C}_r^*(\Gamma,\sigma)$ has nuclear dimension at most $m$. Thus, it follows from \cite{Rob11MComp} that it has $m$-comparison. Thus, $\tau (\mathcal{E}_\Gamma )<1$ implies $\mathcal{E}_\Gamma$ is isomorphic to a submodule of $A^{m+1}$. Using again the results of \cite[Proposition 5.2]{BeEnva22} there exists an $(m+1)$-multiwindow frame with windows in $\mathcal{H}_\pi^\infty$.
\end{proof}

We end with a concrete example where \Cref{prp:ThmEMain} gives new sufficient conditions for the existence of smooth frames and Riesz sequences.

\begin{example}
We consider \cite[Example 7.2]{BeEnva22}. Set $G = \R \times H_3(\R)$ where $H_3(\R)$ is the 3-dimensional real Heisenberg group. We will use the group law on $G$ given by
\[ (s_1,s_2,s_3,s_4)(t_1,t_2,t_3,t_4) = (s_1+t_1,s_2+t_2+s_4t_3, s_3+t_3,s_4+t_4) . \]
There is a projective, square-integrable, irreducible, unitary representation $\pi$ of $G$ on $L^2(\R^2)$ given by
\[ \pi(s_1, \ldots, s_4)f(x, y) = e^{2\pi i (s_3 x - s_2 y + s_3 y^2/2)} f(x-s_1,y- s_4) . \]
The associated 2-cocycle is given by
\[ \sigma((s_1, \ldots, s_4),(t_1, \ldots, t_4)) = e^{2\pi i (-s_1t_3 + s_4 t_2 + s_4^2 t_3 /2)} . \]
A Haar measure on $G$ can be chosen as the 4-dimensional Lebesgue measure on the underlying Euclidean space, and with this choice $d_\pi$ equals $1$.

Let $\alpha, \beta > 0$. A family of lattices in $G$ are given by
\[ \Gamma_{\alpha,\beta} = \{ (\alpha k_1, \beta^2 k_2, \beta k_3, \beta k_4) : k_1,k_2,k_3,k_4 \in \Z \} . \]
As abstract groups these are all isomorphic to $\Gamma = \Z \times H_3(\Z)$ where $H_3(\Z)$ denotes the integer Heisenberg group. To simplify notation we consider instead $\Gamma$ with the 2-cocycle
\[ \sigma_{\alpha,\beta}((k_1,\ldots,k_4),(l_1, \ldots, l_4)) = e^{2\pi i (- \alpha \beta k_1 l_3 + \beta^3 k_4 l_2 + \beta^3 k_4^2 l_3/2)} . \]
Now
\[ \mathrm{Z}(\Gamma,\sigma_{\alpha,\beta}) = \{ ( k_1, k_2, 0, 0) : \alpha \beta k_1 \in \Z, \beta^3 k_2 \in \Z \}  \]
which is of infinite index in $\Gamma$. Hence, $(\Gamma,\sigma)$ satisfies Kleppner's condition if and only if both $\alpha \beta$ and $\beta^3$ are irrational. The main result of \cite{BeEnva22} then states that in this case there exists $\eta \in \mathcal{H}_\pi^\infty$ such that $\pi(\Gamma_{\alpha,\beta})\eta$ is a frame (resp.\ Riesz sequence) if $\covol(\Gamma_{\alpha,\beta}) < 1$ (resp.\ $\covol(\Gamma_{\alpha,\beta}) > 1$). We will show that the restriction of $\sigma$ to $\Gamma$ is non-rational when either $\alpha \beta$ or $\beta^3$ is irrational. Thus, by \Cref{thm:gfta} we can reach the same conclusion about frames and Riesz sequences over $\Gamma_{\alpha,\beta}$ with this weaker assumption.

Suppose first that $\alpha \beta$ is rational, say $a \alpha \beta \in \Z$ for a minimally chosen $a \in \Z$, and that $\beta^3$ is irrational. Then $\mathrm{Z}(\Gamma,\sigma_{\alpha,\beta}) = a \Z \times \{ 0 \}^3$, so $\Gamma / \mathrm{Z}(\Gamma,\sigma_{\alpha,\beta}) \cong \Z / a \Z \times H_3(\Z)$. In this case a section $c \colon \Gamma/ \mathrm{Z}(\Gamma,\sigma_{\alpha,\beta}) \to \Gamma$ is given by $c([k_1], k_2, k_3, k_4) = (k_1 \; \mathrm{mod} \; a, k_2, k_3, k_4)$. Note also that the formula for $\sigma_{\alpha,\beta}$ gives a well-defined 2-cocycle $\omega$ on $\Gamma/ \mathrm{Z}(\Gamma,\sigma_{\alpha,\beta})$. Since $c$ maps the center of $\Gamma / \mathrm{Z}(\Gamma,\sigma_{\alpha,\beta})$ into the center of $\Gamma$, it follows that $\omega_\gamma = \omega$ for every $\gamma \in \dual{\mathrm{Z}(\Gamma,\sigma_{\alpha,\beta})}$, so $\mathrm{Z}( \Gamma/\mathrm{Z}(\Gamma,\sigma_{\alpha,\beta}), \omega_\gamma) = \{ e \}$ for all $\gamma$. Hence $\sigma_{\alpha,\beta}$ is non-rational.

Suppose next that $\beta^3$ is rational, say $b \beta^3 \in \Z$ for a minimally chosen $b \in \Z$, and that $\alpha \beta$ is irrational. Then $\mathrm{Z}(\Gamma,\sigma_{\alpha,\beta}) = \{ (0, bk_2, 0, 0) : k_2 \in \Z \}$. The commutator of two elements $k,l \in \Gamma$ is given by
\[ [k,l] = (0, k_3 l_4 - k_4 l_3, 0, 0) , \]
so $D \coloneqq \mathrm{C}(\Gamma / \mathrm{Z}(\Gamma, \sigma_{\alpha,\beta})) = \{ [0,k_2,0,0] : 0 \leq k_2 < b \} \cong \Z_{b}$. One then checks that the conditions of \Cref{prp:CharNSCAforSome2step} are satisfied. The kernel of $\varphi_D(\omega)$ is given by $M = \{ [k_1,k_2,k_3,bk_4] : k_i \in \Z, 0 \leq k_2 < b \}$. A section $c \colon \Gamma / \mathrm{Z}(\Gamma,\sigma_{\alpha,\beta}) \to \Gamma$ of the quotient map is given by $c([k_1,k_2,k_3,k_4]) = (k_1, k_2 \; \mathrm{mod} \; b, k_3, k_4)$. Letting $\gamma$ be a character on $\mathrm{Z}(\Gamma,\sigma_{\alpha,\beta})$, say $\gamma(0,k_2,0,0) = e^{2\pi i \xi k_2}$ for some $\xi \in \R$, we then have that
\[ \widetilde{\omega_\gamma}([k],[l]) = e^{2\pi i \xi (k_4 l_3 - k_3 l_4)} e^{2\pi i (-\alpha\beta (k_1 l_3 - k_3 l_1) + \beta^3 (k_4 l_2 - l_4 k_2) + \beta^3 (k_4^2 l_3 - l_4^2 k_3)/2) } . \]
The elements of the center of $M$ are of the form $[k_1,k_2,bk_3,bk_4]$ for $k_i \in \Z$, and such an element is $\mathrm{Res}(\omega_\gamma)$-regular precisely when
\[ \xi( bk_4 l_3 - b^2 k_3 l_4) - \alpha \beta( k_1 l_3 - b k_3 l_1) + \beta^3 (b k_4 l_2 - b l_4 k_2) + \beta^3 (b^2 k_4^2 l_3 - b^3 l_4^2 k_3)/2 \in \Z \]
for all $l_i \in \Z$. Setting $l_1 = 1$ and $l_i = 0$ for $i \neq 1$ gives that $\alpha \beta b k_3 \in \Z$, which forces $k_3 = 0$ since $\alpha \beta$ is irrational. This means that $\mathrm{Z}(M, \mathrm{Res}(\omega_\gamma))$ is a subgroup of $\{ [k_1, k_2, 0, bk_4] : 0 \leq k_2 < b, k_i \in \Z \}$ which already has infinite index in $M$. We can therefore conclude that $[M : \mathrm{Z}(M, \mathrm{Res}(\omega_\gamma)) ] = \infty$, which by \Cref{prp:CharNSCAforSome2step} shows that $\sigma_{\alpha,\beta}$ is non-rational. This finishes the argument.
\end{example}

\printbibliography

\end{document}